\documentclass{article}
\usepackage[utf8]{inputenc}
\usepackage{amsthm,amsmath}

\usepackage{amssymb}
\usepackage{natbib} 
\usepackage{authblk}									

\title{Bayesian quadrature for 
$\Hone$
with Poincaré inequality on a compact interval}
\author[1]{Olivier Roustant}
\author[2]{Nora L\"uthen}
\author[3]{Fabrice Gamboa}
\affil[1]{UMR CNRS 5219, Institut de Mathématiques de Toulouse, INSA, Université
de Toulouse, France}
\affil[2]{Chair of Risk, Safety, and Uncertainty Quantification, ETH Z\"urich,
8093 Z\"urich, Switzerland}
\affil[3]{UMR CNRS 5219, Institut de Mathématiques de Toulouse, Université
de Toulouse, France}

\usepackage{graphicx}
\usepackage{xcolor}
\usepackage[caption=false]{subfig}
\usepackage[normalem]{ulem} 
\usepackage{hyperref}
\usepackage[margin=25mm]{geometry}
\usepackage{placeins} 
\usepackage{bbold}
\usepackage{comment}

\newcommand{\Ha}{\mathcal{H}}
\newcommand{\Hk}{\mathcal{H}_K}
\newcommand{\HkM}{\mathcal{H}_{K_M}}
\newcommand{\N}{\mathbb{N}}
\newcommand{\R}{\mathbb{R}}

\newcommand{\HH}[1]{H^{#1}(\mu)}
\newcommand{\Hone}{\HH{1}}

\newcommand{\Var}{\mathbb{V}\text{ar}}
\newcommand{\wce}{\mathrm{wce}}

\newcommand{\bound}{\mathcal{B}}

\numberwithin{equation}{section}
\theoremstyle{plain}
\newtheorem{thm}{Theorem}
\newtheorem{lem}{Lemma}
\newtheorem{defi}{Definition}
\newtheorem{prop}{Proposition}
\newtheorem{cor}{Corollary}
\newtheorem{rem}{Remark}
\newtheorem{example}{Example}
\newtheorem{assum}{Assumption}

\begin{document}

\maketitle

\begin{abstract}
Motivated by uncertainty quantification of complex systems, we aim at finding quadrature formulas of the form $\int_a^b f(x) d\mu(x) = \sum_{i=1}^n w_i f(x_i)$ where $f$ belongs to $\Hone$. Here, $\mu$ belongs to a class of continuous probability distributions on $[a, b] \subset \R$ and $\sum_{i=1}^n w_i \delta_{x_i}$ is a discrete probability distribution on $[a, b]$. 
We show that $\Hone$ is a reproducing kernel Hilbert space with a continuous kernel $K$, which allows to reformulate the quadrature question as a Bayesian (or kernel) quadrature problem. Although $K$ has not an easy closed form in general, we establish a correspondence between its spectral decomposition and the one associated to Poincaré inequalities, whose common eigenfunctions form a $T$-system \citep{karlin1966t}. 
The quadrature problem can then be solved in the finite-dimensional proxy space spanned by the first eigenfunctions.
The solution is given by a generalized Gaussian quadrature, which we call Poincaré quadrature.\\ 
We derive several results for the Poincaré quadrature weights and the associated worst-case error. 
When $\mu$ is the uniform distribution, the results are explicit: the Poincaré quadrature is equivalent to the midpoint (rectangle) quadrature rule. Its nodes coincide with the zeros of an eigenfunction 
and the worst-case error scales as $\frac{b-a}{2\sqrt{3}}n^{-1}$ for large $n$.
By comparison with known results for $H^1(0,1)$, this shows that the Poincaré quadrature is asymptotically optimal.
For a general $\mu$, we provide an efficient numerical procedure, based on finite elements and linear programming.
Numerical experiments provide useful insights: nodes are nearly evenly spaced, 
weights are close to the probability density at nodes, and the worst-case error is approximately $O(n^{-1})$ for large $n$. 

\end{abstract}

\tableofcontents

\section{Introduction}

\paragraph{Motivation.} This research is motivated by uncertainty quantification of complex systems, where a typical task is to compute integrals $I = \int G(x) dP(x)$. Here, $G$ is a multivariate function representing a quantity of interest of the system,
$x$ is a vector of $\R^d$ representing the input variables and $P$ is a probability distribution representing the uncertainty on $x$. In this context, the evaluation of $G$ is often time-consuming, and cubature formula may be preferred to sampling techniques to compute the integral $I$. Assuming that the input variables are independent, cubature formulas then boil down to $1$-dimensional quadrature formulas, by tensorization or using sparse grids. 

\paragraph{Problem considered.}
For a given interval $[a,b]$ of $\R$, we aim at finding accurate approximations of  integrals $\int_a^b f(x) d\mu(x)$ when $\mu$ is replaced by a discrete probability distribution. We thus consider quadrature formulas 
\begin{equation} 
\label{eq:quadratureDef}
\int_a^b f(x) d\mu(x) = \sum_{i = 1}^n w_i f(x_i),
\end{equation}
where, for $i=1,\cdots, n$, the quadrature nodes $x_i$ lie in $[a, b]$. The quadrature weights $(w_i)$ are non-negative and sum to $1$. We will denote by $X$ (resp. $w$) the sequence of nodes $(x_i)$ (resp. of weights $(w_i)$).
Considering minimal regularity conditions, we assume that $f$ belongs to the Sobolev space $\Hone = \{ f \in L^2(\mu), s.t. f' \in L^2(\mu) \}$, where the derivatives are defined in a weak sense. More generally, for an integer $p\geq 2$, we define the Sobolev space $\HH{p}$ as the subset of functions $f$ of $\HH{p-1}$ such that $f^{(p)}\in L^2(\mu)$. 
In the whole paper, we consider the usual norm of $H^1(\mu)$ defined by 
$\Vert f \Vert_{\Hone}^2 = \Vert f \Vert^2 + \Vert f' \Vert^2$, 
where $\Vert f \Vert^2 = \int_a^b f^2(x)d\mu(x)$
is the norm of $L^2(\mu)$.
For technical reasons, we assume that $\mu$ is a \emph{bounded perturbation} of the uniform distribution on $[a, b]$, meaning that it admits a continuous probabity density function $\rho$ that does not vanish on $[a,b]$. This includes a wide range of probability distributions used in practice, such as the truncated normal, obtained by conditioning a Gaussian variable to vary in a finite domain. This assumption implies that the sets $L^2(\mu), H^1(\mu)$ contain the same equivalence classes of functions than $L^2(a,b), H^1(a,b)$, associated to the uniform distribution on $[a,b]$, with equivalent norms.

\paragraph{Bayesian -- or kernel -- quadrature formulation.}
When $\mu$ is the uniform probability distribution, it is well known that $H^1(a,b)$ is a reproducing kernel Hilbert space (RKHS). 
Under the previous assumption on $\mu$, we will show that $H^1(\mu)$ is also a RKHS (Section~\ref{sec:PoincRKHS}). In that case, a suitable criterion to evaluate the accuracy of a quadrature $(X, w)$ is the worst-case error, defined by
$$ \wce(X, w, \mathcal{H}) = \sup_{h \in \mathcal{H}, \Vert h \Vert_{\mathcal{H}} \leq 1} \left\vert \int h(x) d\mu(x) - \sum_{i= 1}^n w_i h(x^i) \right\vert.$$
Here, $\mathcal{H}$ is some particular given functional space. Indeed, when $\mathcal{H}$ is a RKHS, $\wce(X, w, \mathcal{H})$ can be explicitly computed as a function of the kernel $K$ (Section~\ref{sec:kernel_quadrature}). 
An interesting  quadrature problem is so the following minimization problem:
$$ (P): \qquad \min_{X, w} \wce(X, w, \mathcal{H}),$$
where one wish to identify the minimizing quadrature. 
Such problem is often called kernel quadrature, or Bayesian quadrature, as the prior information is that the functions lie in the RKHS associated to the kernel $K$.

\paragraph{Originality of the problem.}
We remark that, apart from the case of the uniform distribution, the problem does not reduce to the more standard quadrature problem with a weight function
\begin{equation} \label{eq:quadWeight}
\int_a^b f(x) \rho(x) dx = \sum_{i=1}^n w_i f(x_i)
\end{equation}
where $f$ belongs to $H^1(a,b)$. 
Indeed, the unit balls 
$\{ h \in \Hone, \Vert h \Vert_{\Hone} \leq 1 \}$
and 
$\{ h \in H^1(a,b), \Vert h \Vert_{H^1(a,b)} \leq 1 \}$ 
are different if $\rho$ is not a constant function. Thus $\wce(X, w, \Hone) \neq \wce(X, w, H^1(a,b))$ and the weighted quadrature problem, formulated as a worst-case error minimization problem, will in general not give the same solutions as $(P)$.

\paragraph{Problem resolution in a finite-dimensional proxy space.}
A difficulty in our frame is that the kernel $K$ is in general not known explicitly. Thus $(P)$ cannot be solved directly. 
A key result of this paper is that there is a correspondence between the spectral decomposition of $K$ and the one associated to Poincar\'e inequalities. Furthermore, the common eigenfunctions form a Tchebytchev system ($T$-system, see \citep{karlin1966t}). This has two main consequences.
Firstly, one can compute numerically the spectral decomposition of $K$ with a finite element technique \citep{PoincInterval}. 
Secondly, the worst-case problem $(P)$ can be replaced by a tractable proxy problem 
$$ (P_M): \qquad \min_{X, w} \wce(X, w, \mathcal{H}_M)$$
where $\Hone$ has been replaced by its projection $\mathcal{H}_M$ onto the space spanned by the first $M$ eigenfunctions. Indeed, similarly to polynomials, $T$-systems admit a Gaussian quadrature and for a given number of nodes $n$, there exists a unique quadrature $(X, w)$ with positive weights for which  $\wce(X, w, K_M) = 0$, where $M=2n-1$ is maximal. We call this optimal quadrature \emph{Poincaré quadrature}.
For a general probability distribution $\mu$, the Poincaré quadrature is computed efficiently by linear programming.

\paragraph{Properties of the Poincaré quadrature.}
We derive several results for the connection between the kernel associated to $\Hone$, the Poincaré quadrature nodes and weights, and the associated worst-case error. 
When $\mu$ is the uniform distribution, the results are explicit (Section~\ref{sec:PoincQuadPropUnif}): the Poincaré quadrature is equal to the midpoint (rectangle) quadrature rule, its nodes coincide with the zeros of an eigenfunction, as for the Gaussian quadrature of polynomials, and the worst-case error scales as
$\frac{b-a}{2\sqrt{3}}n^{-1}$ for large $n$.
Furthermore, in the case of $H^1(0,1)$, the kernel is given explicitly, and it is possible to compute the optimal kernel quadrature for it and not only for its finite-dimensional approximation. The results obtained by \cite{these_duc_jacquet_1973} show that the optimal kernel quadrature has evenly space nodes and weights asymptotically equal to $\frac{1}{n}$, which shows that the Poincaré quadrature is asymptotically optimal.

In the general case, numerical experiments provide empirical insights (Section~\ref{sec:NumExp}): nodes are nearly evenly spaced, weights are close to the probability density at nodes, and the worst-case error is approximately proportional to $n^{-1}$ for large $n$. 

\paragraph{Links with literature.}
To the best of our knowledge, considering Sobolev spaces with a non-uniform probability distribution is new. As mentioned above, this does not boil down to a quadrature with weights for the uniform distribution, as the unit balls are different. The case of $H^1(0,1)$ (uniform case) has been studied by several authors, with with various choices of norms and weight functions (Equation~\ref{eq:quadWeight}).
For instance, \cite{zhang_novak_2019optimal} provide expressions of the radius of information (worst-case error for the optimal quadrature) in function of the nodes, for the semi-norm $\int_0^1 f'(x)^2 dx$
and centered weight functions.
For a constant weight function, and the usual norm of $H^1(0, 1)$ considered in the present paper,  \cite{these_duc_jacquet_1973} obtains the optimal kernel quadrature. The link between $T$-systems and kernel quadrature has been also exploited in \cite{oettershagenPhD2017}. There, the kernel is assumed to have an explicit form, and the $T$-system is obtained by considering the kernel function at nodes $K(x_i, .)$, which is different than our approach based on the spectral decomposition of $K$. The case of $H^1(0,1)$ is considered in their numerical experiments, but with a different norm associated to Bernoulli polynomials, given by $\Vert f \Vert^2 = (\int_0^1 f(x)dx)^2 + \int_0^1 f'(x)^2 dx$. 

\paragraph{Paper organization.}
Section~\ref{sec:background} gives the prerequisites on Poincaré inequalities, $T$-systems, RKHS and kernel quadrature. 
The analysis of the RKHS structure of $H^1(\mu)$ is done in Section~\ref{sec:PoincRKHS}, where a connection is established between Poincaré inequalities and the kernel of $\Hone$. 
Section~\ref{sec:PoincQuadPropGen} gives general formulas for the optimal quadrature weights and the associated worst-case error as a function of the kernel. Section~\ref{sec:PoincQuadPropUnif} focuses on the case of the uniform distribution. 
Section~\ref{sec:NumExp} presents numerical experiments in the general case. 

\section{Background}
\label{sec:background}

In the whole paper, we consider a bounded interval of the real line $[a, b]$, with $-\infty < a < b< \infty$. 
We consider a probability distribution $\mu$ supported on $[a, b]$ which is a bounded perturbation of the uniform distribution, in the following sense.

\begin{defi}[Bounded perturbation of the uniform distribution]
Let $\mu$ be a continuous probability distribution on $[a,b]$, with density $\rho$. We say that $\mu$ is a \emph{bounded perturbation of the uniform distribution}
if $\rho$ is a positive continuous and piecewise $C^1$ function on $[a,b]$.
We denote by $\bound$ the set of bounded perturbations of the uniform distribution on $[a,b]$.\\
We also denote by $V = -\log(\rho)$ the so-called \emph{potential} associated to $\mu$. Equivalently, $\rho(t) = e^{-V(t)}.$
\end{defi}
{\bf Remarks}
\begin{itemize}
\item Obviously, if $\mu$ fulfils the previous definition, then $\rho$ is bounded from below and above by positive constants:
there exist $m, M$ in $\R$ such that
$$ \forall t \in [a,b], \qquad  0 < m \leq \rho(t) \leq M < +\infty.$$
\item When $\mu \in \bound$, it is straightforward that the sets $L^2(\mu), H^1(\mu)$ contain the same equivalence classes of functions than $L^2(a,b), H^1(a,b)$, associated to the uniform distribution on $[a,b]$, with equivalent norms.
\end{itemize}

\subsection{Poincaré inequalities and basis}

This section is based on \cite{PoincInterval} \cite[see also][]{BGL_book}.
Let $\mu$ be a probability distribution on $[a, b]$. For $f, g \in L^2(\mu)$, let $ \Vert f \Vert = \left( \int f^2 d\mu \right)^{1/2} $ be the usual norm, and $\langle f, g \rangle = \int fg d\mu$ the usual dot product. Denote by
$\Var_\mu(f)$ the variance of $f$:
$$ \Var_\mu(f) :=  \left\Vert f - \int f d\mu \right\Vert^2.$$
We first recall the notion of Poincar\'e inequality.

\begin{defi}[Poincar\'e inequality]
We say that $\mu$ verifies a Poincar\'e inequality if there exists a
finite constant $C$ such that for all $f \in \Hone$:
$$\Var_\mu(f) \leq C \Vert f' \Vert^2.$$
In this case, the smallest possible constant $C$ above is denoted $C_P(\mu)$, and is called Poincaré constant of $\mu$.
\end{defi}

When it exists, the Poincar\'e constant is obtained by minimizing the so-called Rayleigh ratio $J(f) = \frac{\Vert f' \Vert^2}{\Vert f \Vert^2}$ over all centered functions of $\Hone$.
An important result is that a bounded perturbation of the uniform distribution admits a Poincar\'e inequality, which is related to a spectral decomposition:

\begin{thm}[Spectral theorem] 
\label{thm:spectralThm}
Let $\mu$ be a probability distribution in $\bound$.
Consider the following problems:
\begin{itemize}
\item[(P1)] $\textrm{Find } f \in \Hone \textrm{ s.t. } \quad
J(f) = \frac{\Vert f' \Vert^2}{\Vert f \Vert^2} \quad \textrm{ is minimum under} \quad \int f d\mu=0.$ 
\item[(P2)] $\textrm{Find } f \in \Hone \textrm{ s.t. } \quad
\langle f', g' \rangle = \lambda \langle f, g \rangle \quad \forall g \in \Hone.$ 
\item[(P3)] $\textrm{Find } f \in H^2(\mu) \textrm{ s.t. } \quad f'' - V'f' = - \lambda f \quad \textrm{ and } \quad f'(a) = f'(b) = 0.$ 
\end{itemize}
Then the eigenvalue problems (P2) and (P3) are equivalent, and their eigenvalues form an increasing sequence $(\lambda_m)_{m \geq 0}$ of non-negative real numbers that tends to infinity. They are all simple, and $\lambda_0 = 0$. The eigenvectors $(\varphi_m)_{m \geq 0}$ form a Hilbert basis of $L^2(\mu)$, and $\varphi_0$ is a constant function.\\
Furthermore when $\lambda = \lambda_1$, the first positive eigenvalue, $(P2)$ and $(P3)$ are equivalent to $(P1)$ and the minimum of $(P1)$ is attained for $f = \varphi_1$. 
Thus $C_P(\mu)= 1/\lambda_1$.
\end{thm}

In this paper, our interest is in the whole spectral decomposition. In particular, we define the Poincar\'e basis of $L^2(\mu)$ as follows.

\begin{defi}[Poincar\'e basis]
Let $\mu$ be a probability distribution in $\bound$. We call Poincar\'e basis an orthonormal basis formed by eigenfunctions $(\varphi_m)_{m \geq 0}$ of the spectral theorem (Theorem~\ref{thm:spectralThm}).
As all eigenvalues are simple, a Poincar\'e basis is unique up to a sign change for each eigenfunction. We set $\varphi_0 = 1$.\\
\end{defi}

We conclude this section by a link to the Sturm-Liouville theory of second-order differential equations.

\begin{prop} \label{prop:PoincareAndSL}
Let $\mu$ be a probability distribution in $\bound$.
Then the Poincar\'e basis consists of the eigenfunctions of the Sturm-Liouville eigenproblem
\begin{equation}  \label{eq:PoincSturmLiouville}
L(f)(x) = \beta r(x) f(x)
\end{equation}
with Neumann conditions $f'(a)=f'(b)=0$, where $L(f)(x) = - (p(x)f'(x))' + q(x) f(x)$ and $q = r = p = e^{-V}.$
Furthermore, the Sturm-Liouville problem is regular, in the sense that all eigenvalues are  positive.
\end{prop}

\begin{proof}
Recall that $\langle f, g \rangle_{\Hone} = \langle f, g \rangle + \langle f', g' \rangle$. From the proof of Theorem 2 in \cite{PoincInterval}, the eigenfunctions of the Poincaré operator are solutions of the spectral problem: to find $f \in \Hone$ and $\beta$ such that for all $g \in \Hone$, 
\begin{equation} \label{eq:PoincWeakForm1}
\langle f, g \rangle_{\Hone} = \beta \langle f, g \rangle.
\end{equation}
The corresponding eigenvalues are $\beta_m = 1+\lambda_m$. 
In particular, $\beta_m > 0$, as $\lambda_m \geq 0$. 
Moreover, Problem (\ref{eq:PoincWeakForm1}) is equivalent to the second order differential equation 
$$ f''(x) - V'(x)f'(x) - f(x) = - \beta f(x) $$
with Neumann conditions $f'(a) = f'(b) = 0$  \cite[see also][proof of Theorem 2]{PoincInterval}. 
Multiplying by $e^{-V(x)}$, which by definition of $\bound$ does not vanish, we obtain the equivalent Sturm-Liouville form (\ref{eq:PoincSturmLiouville}).
\end{proof}

\subsection{Quadrature with T-systems}
\label{sec:1Dquadrature}

This section is based on \cite{karlin1966t}.

\begin{defi}[T-systems, generalized polynomials] 
Let $(u_n)_{n \in \N}$ be a family of real-valued continuous functions defined on a compact interval $[a, b] \subset \R$. We say that $(u_n)$ is a \emph{complete Tchebytchev system}, or simply \emph{T-system}, if for all integer $n \geq 1$ and for all sequence of distinct points $a \leq t_1 < \dots < t_n \leq b$, the determinant of the \emph{generalized Vandermonde matrix} 
$$ V(u_0, \dots, u_{n-1}; t_1, \dots, t_n) :=  
\begin{pmatrix}
u_0(t_1) & \dots & u_0(t_n)  \\
\vdots & \ddots & \vdots \\
u_{n-1}(t_1) & \dots & u_{n-1}(t_n) \\
\end{pmatrix}
$$
is positive. The finite linear combinations of the $u_n$'s are called \emph{generalized polynomials} or \emph{u-polynomials}.
\end{defi}

A prototype of $T$-system on any interval of the real line is given by the polynomial functions $u_i(t) = t^i$, and $V$ is equal to the Vandermonde determinant $\det V(1, t, \dots, t^{n-1}, t_1, \dots, t_n) = \prod_{1 \leq i < j \leq n} (t_j - t_i)$ 
\cite[see e.g.][page 1]{karlin1966t}. 
An equivalent definition of $T$-systems (up to a sign change) is that any generalized polynomial, i.e., any linear combination of $u_0, u_1, \dots, u_n$, has at most $n$ zeros 
\cite[][Theorem 4.1.]{karlin1966t}.
This extends the property that (ordinary) polynomials of degree $n$ have at most $n$ zeros.
In that sense, $T$-systems can be viewed as a generalization of polynomials, which justifies the name {\it generalized polynomials}.\\

In the context of quadrature problems, the definition of $T$-systems guarantees that for any set of distinct quadrature nodes in $[a, b]$, there exists a unique set of quadrature weights such that the quadrature formula (\ref{eq:quadratureDef})
is exact at order $n-1$, i.e. for all functions in  $\text{span}(u_0, \dots, u_{n-1})$. Indeed, up to reordering, the equations above define a linear system whose matrix is invertible and equal, up to a sign change, to $V(u_0, \dots, u_{n-1}; t_1, \dots, t_n$). This quadrature formula thus generalizes the Newton-Cotes quadrature of polynomials, and suffers in general from the same drawback: the weights can be negative and the resulting quadrature formula can be instable.\\

Interestingly, extending the Gaussian quadrature of ordinary polynomials to more general functions, $T$-systems admit a unique quadrature that has positive weights and is exact at order $2n-1$. Contrarily to the polynomial case, however, the nodes of this quadrature do in general not coincide with the zeros of a (generalized) polynomial. The computation of the nodes and weights uses a different approach, relying on geometry. More precisely, consider the \emph{moment space}:
\begin{multline} \label{eq:MomentSpace}
    \mathcal{M}_{n+1} = \bigg\{ c \in \mathbb{R}^{n+1},  \textrm{ with } c_i = \int_a^b u_i(t) d\sigma(t), \textrm{ for all } i=0, \dots, n, \\
  \text{where $\sigma$ is a finite measure with support $[a,b]$} 
    \bigg\}.
\end{multline}
It can be shown that the moment space is a closed convex set. Then, we have the announced result.

\begin{prop} \label{prop:GaussQuadTsystem}
Let $u=(u_n)_{n \in \N}$ be a $T$-system, and $\mu$ be a probability distribution in $\bound$.
Then, for all $n \in \N$, the vector $c := \left( \int_a^b u_i(t) d\mu(t) \right)_{0 \leq i \leq 2n-1}$ is an interior point of $\mathcal{M}_{2n}$, and
there exists a unique quadrature  (\ref{eq:quadratureDef}) with positive weights which is exact at order $2n-1$ (i.e. exact on the vector space spanned by $u_0, \dots, u_{2n-1}$) and uses a minimal number of nodes, which is equal to $n$.
Its nodes are all in the open interval $(a, b)$ and its weights sum to $1$.\\
It coincides with the Gaussian quadrature when the $T$-system is formed by polynomials $u_i(t) = t^i$.\\
Furthermore, this quadrature is obtained by solving the minimization problem 
\begin{equation} \label{eq:RepRec}
    \underset{\sigma \in V_{2n-1}(c)}{\min} \int_a^b u_{2n}(t) d\sigma(t)
\end{equation}
over the set $V_{2n-1}(c)$ of probability distributions subject to moment conditions $\int_a^b u_i(t) d\sigma(t) = c_i$, for $i=0, 1, \dots, 2n-1$.
\end{prop}

\begin{proof}[Proof of Proposition~\ref{prop:GaussQuadTsystem}]
The proof is based on different results given in \cite{karlin1966t}, that we will refer to. 
Let us denote $N=2n-1$ the quadrature order.\\
First, from Lemma 9.2., page 65, $c$ is an interior point of $\mathcal{M}_N$ if and only if for all non-zero $u$-polynomial $g = \sum_{i=0}^N a_i u_i$ such that $g(t) \geq 0$ for all $t$ in $[a, b]$, then $m_g = \sum_{i=0}^N a_i c_i >0$. Clearly, $m_g = \int_a^b g d\mu$ and thus $m_g \geq 0$. Assume that $m_g=0$. Let us write $m_g = \int_a^b g(x) \rho(x) dx$. As $g, \rho$ are continuous and non-negative, it implies that $g \rho = 0$ on $[a, b]$. As $\rho = e^{-V}$ is non-vanishing on $[a, b]$, it implies that $g$ is identically zero on $[a, b]$, which is contradictory. Finally $m_g > 0$, which shows that $c$ is an interior point of $\mathcal{M}_N$.\\
Now, follow \cite{karlin1966t}, \S 3, case (ii), page 46.
They use the notion of index of a sequence of nodes, defined by the number of nodes in $[a, b]$, with a half weight for the nodes equal to the endpoints $a, b$ (if any). Then, as $N$ is odd, they show that there are exactly two quadratures with positive weights and the smallest possible index, equal to $(N+1)/2 = n$. These quadrature are called \emph{principal representations} in this context. For one of them, called \emph{upper principal representation}, 
the nodes include the endpoints, and thus the quadrature involves $n+1$ nodes formed by the endpoints $a, b$ and $n-1$ nodes in $(a, b)$.
The other one, called \emph{lower principal representation}, involves $n$ nodes in the open interval $(a,b)$. Thus, it is the only quadrature with positive weights and containing the smallest number of nodes, equal to $n$.\\
Furthermore, it is shown in Theorem 1.1. page 80, that, if $u$ is a $T$-system, the solution of (\ref{eq:RepRec}) is unique and equal to the lower representation of $u$. In particular, the weights are positive and sum to one.\\
Finally, if $u_i(t)=t^i$, we have equality with the Gaussian quadrature by uniqueness of the quadrature, since the Gaussian quadrature has $n$ distinct nodes in $(a,b)$, positive weights summing to one, and is exact for $u_0, \dots, u_{2n-1}$ \citep[][Chapter IV]{karlin1966t}.
\end{proof}

\begin{rem}
Replacing the minimization problem in (\ref{eq:RepRec}) by maximization, we obtain another valid quadrature of oder $2n-1$, called \emph{upper principal representation}. However, it involves one more node, i.e. $n+1$ nodes, including the endpoints $a, b$. Equivalently, for a fixed number $n$ of nodes, this quadrature has order $2n-3$, compared to $2n-1$ for the Gaussian quadrature. It generalizes the Lobatto quadrature for polynomials. 
\end{rem}

\begin{defi}[Gaussian quadrature for $T$-systems]
\label{def:GaussQuadTsystem}
The unique quadrature of Prop.~\ref{prop:GaussQuadTsystem}
is called \emph{lower principal representation} of $u$.
By analogy with polynomials, and following \cite{oettershagenPhD2017},  we will  call it generalized Gaussian quadrature, or simply \emph{Gaussian quadrature} of the $T$-system $(u_n)_{n \in \N}$. \end{defi}

We now show that a Poincaré basis is a $T$-system. This is an immediate consequence of the example of eigenfunctions of Sturm-Liouville problems 
\cite[see e.g. Example 7 in][]{karlin1966t}.
A proof can be found in \cite{oscillationBook}. However, it is not easy to read as it is split in several parts. We recall below the main steps and give a roadmap for the interested reader.

\begin{prop} \label{prop:PoincareTsystem}
Consider the notations and the assumption of Theorem~\ref{thm:spectralThm}. Then the Poincaré basis $(\varphi_m)_{m \in \N}$ is a $T$-system.
\end{prop}

\begin{proof}
By Prop.~\ref{prop:PoincareAndSL}, the eigenfunctions of the Poincaré operator are eigenfunctions of the regular Sturm-Liouville problem (\ref{eq:PoincSturmLiouville}).
Then the result is a particular case of a more general result stating that the eigenfunctions of a regular Sturm-Liouville operator form a $T$-system. A proof can be found in \cite{oscillationBook}. We give here the three main steps, and pointers to the corresponding sections.\\
    Firstly, it is proved in IV.10.4, (pages 236 -- 238), that the eigenfunctions of a Sturm-Liouville operator under boundary constraints verify an integral equation of the form
    \begin{equation} \label{eq:SturmLiouville_IntEq}
    \varphi(x) = \lambda \int_a^b K(x,s) \varphi(s) d\sigma(s)
    \end{equation}
    where $\lambda >0$, $\sigma$ is a probability distribution, and $K$ is the so-called Green function of $L$, defined by:
    $$ K(x,s) = \psi(\min(x,s)) \chi(\max(x, s)),$$
    where $\psi, \chi$ are particular solutions of the homogeneous equations $L(f)=0$, such that $\frac{\psi}{\chi}$ is a non-decreasing function.\\
    Secondly, it is proved that $K$ is an \emph{oscillatory kernel}, in the sense of\footnote{see also Definition 1' page 179, and definitions 4 and 6, pages 74 and 76.} Definition 1 page 178. Roughly speaking, it means that every matrix extracted from $K$ \emph{in ascending order}, i.e. $(K(x_i, s_j))_{1 \leq i, j \leq n}$ with $x_1 < \dots < x_n$ and $s_1 < \dots < s_n$ is positive semidefinite.
    The proof starts at page 78 (Example 5, `Single-pair' matrices), continues at page 103 (Theorem 12), page 220 (Criterion A) and ends at page 238 (Theorem 16).\\
    Thirdly, if $K$ is an oscillatory kernel, then the solutions of the integral equation (\ref{eq:SturmLiouville_IntEq}) form a $T$-sytem, which is proved in Theorem 1, page 181.
\end{proof}

\subsection{Kernel quadrature}
\label{sec:kernel_quadrature}
\paragraph{RKHS.}
We first recall some facts on \emph{reproducing kernel Hilbert spaces} (RKHS), refering to \cite{berlinetRKHSbook} for more details.
For a given set $T$, let $\Ha$ be a Hilbert space of functions $T \to \R$, with norm $\Vert . \Vert$. 
We say that $\Ha$ is a \emph{RKHS} if for all $x \in T$, the evaluation functions 
$h \in \Ha \mapsto h(x)$ are continuous. 
It can be shown that a RKHS is in bijection with a semi-definite positive function, also called \emph{kernel}. 
If $K$ is a kernel associated to $\Ha$, we write $\Ha= \Hk$. 
The RKHS $\Hk$ is characterized by the so-called \emph{reproducing property}
$$ \forall x \in T, \forall h \in \Hk, \qquad 
\langle K(x, .), h \rangle = h(x).$$
In particular, choosing $h = K(y, .)$, we get 
$$ \forall x, y \in T, \qquad 
\langle K(x, .), K(y, .) \rangle = K(x, y).$$

\paragraph{Worst-case error in RKHS.}
In RKHS, worst-case quantities of linear functionals can be computed explicitly. 
Indeed, for instance, the Cauchy-Schwartz inequality gives for all $f \in \Ha$:
$$ \vert h(x) \vert = \vert \langle K(x, .), h \rangle \vert
\leq \Vert K(x, .) \Vert \Vert h \Vert,$$
from which it is deduced immediately $ \sup_{h \in \mathcal{H}_K, \Vert h \Vert \leq 1} \vert h(x) \vert = \Vert K(x, .) \Vert$. Furthermore, by the reproducing property, we have $\Vert K(x, .) \Vert^2 = K(x, x)$, and finally
$$ \sup_{h \in \mathcal{H}_K, \Vert h \Vert \leq 1} \vert h(x) \vert = \sqrt{K(x,x)}.$$
A similar computation can be done for linear functionals defined by quadrature formulas. Let us first define the worst-case error of a general quadrature.

\begin{defi}[worst-case error of a quadrature]
Let $(X,w)$ be a quadrature composed of a set of nodes $X = (x_1, \dots, x_n) \in [a,b]^n$ and a set of weights $w= (w_1, \dots, w_n) \in \R^n$. Let $\Ha$ be a set of functions on $[a,b] \to \R$. The worst-case error of $(X, w)$ on $\Ha$ is defined by
$$ \wce(X, w, \Ha) = \sup_{h \in \Ha, \Vert h \Vert \leq 1} \left\vert \int h(x) d\mu(x) - \sum_{i= 1}^n w_i h(x_i) \right\vert.$$
If $\Ha$ is a RKHS $\Hk$, we simply denote $\wce(X, w, K) = \wce(X, w, \Hk)$. 
\end{defi}

By a direct extension of the computation above, we have
$$ \wce(X, w, K) =  
\left\Vert \int K(x, .) d\mu(x) - \sum_{i = 1}^n w_i K(x_i, .) \right\Vert_{\mathcal{H}_K}.$$ 
Using the reproducing property, one obtains the analytical expression:
$$ \wce(X, w, K)^2 = \iint K(x, x')d\mu(x) d\mu(x') - 2 \sum_{i = 1}^n w_i \int K(x_i, x) d\mu(x) + \sum_{i, j} w_i w_j K(x_i, x_j) $$
which can be rewritten in the matricial form
\begin{equation} \label{eq:wceMatricial}
\wce(X, w, K)^2 = w^\top K(X, X) w - 2 \ell_K(X)^\top w + c_K   
\end{equation}
where $K(X,X) = (K(x_i, x_j))_{1 \leq i,j \leq n}$ is the Gram matrix, $\ell_K(X) = (\int K(x_i, x) d\mu(x))_{1 \leq i \leq n}$ is the column vector formed by the primitive function of the kernel at $x_i$ and $c_K = \iint K(x, x') d\mu(x)d\mu(x')$ is a constant.\\

\paragraph{Kernel quadrature.}
A kernel quadrature is obtained by minimizing the worst-case error. We need the following assumption:
\begin{assum} \label{state:GramInvert}
The Gram matrix $K(X,X)$ is invertible when the elements of $X$ are all different. 
\end{assum}
\noindent Under Assumption~\ref{state:GramInvert}, for a given set of nodes $X$ formed by different nodes, then (\ref{eq:wceMatricial}) defines a strictly convex function. Thus, it has a unique minimum, denoted by $w^\star(X, K)$. By solving the first order conditions, we immediately get the exact expression of the vector of optimal weights:
\begin{equation} \label{eq:optimalWeight}
w^\star(X, K) = K(X, X)^{-1} \ell_K(X).    
\end{equation}
After some algebra, we get the corresponding minimal value for the worst case error:
\begin{equation} \label{eq:wceOptimal}
wce(X, w^\star, K)^2 = c_K - \ell_K(X)^\top K(X, X)^{-1} \ell_K(X).
\end{equation}
Kernel quadrature and optimal kernel quadratures can then be defined as follows.
\begin{defi}[Kernel quadrature, optimal kernel quadrature]
Let $X$ be a set of nodes, and assume that Assumption~\ref{state:GramInvert} is verified. Then, the \emph{kernel quadrature} associated to $X$ on $\Hk$ is the quadrature $(X, w^\star(X, K))$ that minimizes the worst-case error $\wce(X, w, K)$ over all sets of weights in $\R^n$. \\
An \emph{optimal kernel quadrature}, if it exists, is a quadrature $(X, w)$ that minimizes the worst-case error $\wce(X, w, K)$ among all quadratures $(X, w)$, or equivalently, that minimizes $\wce(X, w^\star(X, w), K)$ over all sets of nodes $X$.
\end{defi}

\begin{rem}
Notice that the weights of a kernel quadrature are not constrained to be positive, and not constrained to sum to $1$.
\end{rem}

\section{Spectral decomposition of $H^1(\mu)$ with the Poincaré basis}
\label{sec:PoincRKHS}
We show our main result: when $\mu$ is a bounded perturbation of the uniform distribution, then $H^1(\mu)$ is a RKHS whose kernel eigenfunctions coincide with the Poincaré basis. We illustrate this on two examples where explicit computations can be made.

\subsection{Main result}
\begin{prop}[Mercer's representation of $H^1(\mu)$ with the Poincaré basis]
\label{prop:MercerH1}
Assume that $\mu$ is a probability distribution in $\bound$ with support $[a,b]$, and denote by $(\lambda_m, \varphi_m)_{m \in \N}$ the eigenvalues and (normalized) eigenfunctions of the Poincaré operator. Define $\alpha_m = (1+\lambda_m)^{-1}$.
Then $H^1(\mu)$, with its usual Hilbert norm 
$\Vert f \Vert_{H^1(\mu)}^2 = \Vert f \Vert^2 + \Vert f' \Vert^2$, 
is a RKHS. 
Its kernel $K$ is continuous on $[a,b]^2$ and verifies $\int_a^b K(x, y) d\mu(y) = 1$ for all $x \in [a,b]$. 
Its Mercer's decomposition is written
\begin{equation} \label{eq:H1kerMercer}
K(x, y) = \sum_{m=0}^{\infty} \alpha_m \varphi_m(x) \varphi_m(y), \end{equation}
where the convergence is uniform on $[a,b]^2$. Furthermore, $K$ can be computed as
\begin{equation} \label{eq:singlePairKernel}
    K(x, y) = \frac{1}{C} \psi(\min(x, y)) \chi(\max(x, y)),
\end{equation}
where $\psi, \chi$ are two linearly independent solutions of the homogeneous equation $f'' - f'V' -f = 0$ such that $\psi'(a)=0$ and $\chi'(b)=0$, and 
$C = \chi(b) \int_a^b \psi(x) d\mu(x) = \psi(a) \int_a^b \chi(y) d\mu(y)$ is a normalization constant.
\end{prop}

\begin{proof} Under the assumptions on $\mu$, $L^2(\mu)=L^2(a,b)$ with an equivalent norm, and $\Hone = H^1(a,b)$ with an equivalent norm.
Now, it is well known that $H^1(a,b)$ is a RKHS, with an explicit kernel 
\cite[see e.g.][Example 1.4]{atteia}.
Thus, for all $x \in [a,b]$, the evaluation $f \in H^1(a,b) \mapsto f(x)$ is continuous. By equivalence of the norms, $f \in \Hone \to f \in H^1(a,b)$ is continuous. Hence by composition, $f \in \Hone \to f(x)$ is continuous. This shows that $\Hone$ is a RKHS. Let us denote by $K$ its kernel.

The link between $K$ and the Poincaré inequality is visible through the bilinear form $a(f, g) = \langle f, g \rangle_{\Hone}$. Consider the spectral problem: to find $f \in \Hone$ and $\beta$ such that for all $g \in \Hone$, 
\begin{equation} \label{eq:PoincWeakForm}
a(f, g) = \beta \langle f, g \rangle
\end{equation}
From \cite{PoincInterval} [Theorem 2 and its proof] under the assumptions on $\mu$, there exists a countable sequence of solutions, which is given by $\beta_m = 1+\lambda_m$ and $\varphi_m$ $(m \in \N)$. Notice that $\beta_m >0$ for all $m\in \N$.
Furthermore $\beta_m, \varphi_m$ are defined in an unique way (up to a change sign of the eigenfunctions) because the eigenvalues are simple and the eigenfunctions have norm $1$.\\
Now, since $\Hone$ is a RKHS, the functions $K(x, .)$ are dense in $\Hone$ ($x \in [a,b]$). Thus, Problem (\ref{eq:PoincWeakForm}) is equivalent to:
$$\forall x \in [a,b], \quad a(f, K(x, .)) = \beta \langle f, K(x, .) \rangle.$$
By the reproducing property, $a(f, K(x, .)) = \langle f, K(x, .) \rangle_{\Hone} = f(x)$. Hence, (\ref{eq:PoincWeakForm}) is equivalent to:
find $f \in \Hone$ and $\beta$ such that
\begin{equation} \label{eq:PoincGreen}
\forall x \in [a,b], \quad f(x) = \beta \int_a^b K(x, y) f(y) d\mu(y),
\end{equation}
which is equivalent to the spectral decomposition of the Hilbert-Schmidt operator associated to $K$ (recall that $\beta > 0$).\\
Moreover, by Prop.~\ref{prop:PoincareAndSL} and its proof, 
Problem (\ref{eq:PoincWeakForm}) is equivalent to the regular Sturm-Liouville problem~(\ref{eq:PoincSturmLiouville}).
Thus, from \cite[][Section 10, pages 234-238]{oscillationBook}, we also obtain that the solution of (\ref{eq:PoincSturmLiouville}) is equivalent to the solution of (\ref{eq:PoincGreen}).
In this context, $K$ is called \emph{Green function}. But this point of view gives more details, and tells that $K$ is equal to
\begin{equation*} 
    K(x, y) = \frac{1}{C} \psi(\min(x, y)) \chi(\max(x, y)),
\end{equation*}
where $\psi, \chi$ 
are two linearly independent solutions of the homogeneous equation $L(f) = 0$ such that $\psi'(a)=0$ and $\chi'(b)=0$ 
\cite[][section 7]{oscillationBook}.
The constant $C$ is determined such that 
$\int_a^b K(x,y) d\mu(y) = 1$ for all $x \in [a, b]$. Indeed, as the constant function $1$ belongs to $H^1(\mu)$, the RKHS reproducing property gives, for all $x \in [a,b]$:
$$ 1 = 1(x) = \langle 1, K(x, .) \rangle_{H^1(\mu)} = \int_a^b K(x, y) d\mu(y). $$
For instance, choosing $y=b$ or $x=a$, we obtain
$C = \chi(b) \int_a^b \psi(x) d\mu(x) = \psi(a) \int_a^b \chi(y) d\mu(y)$.\\
Now $\psi$ and $\chi$ are continuous, as elements of $\Hone$ (whose functions are equal to those of $H^1(a, b)$). As $\min, \max$ are continuous functions, we obtain, by composition, that $K$ is continuous on $[a,b]^2$. Hence by Mercer's theorem 
\cite[see e.g.][]{berlinetRKHSbook}, 
$K$ is written in terms of the solutions of \eqref{eq:PoincGreen} as
\begin{equation} \label{eq:kH1ortho}
K(x, y) = \sum_{m \in \N} \alpha_m \varphi_m(x) \varphi_m(y),
\end{equation}
with $\alpha_m = \frac{1}{\beta_m} = \frac{1}{1+\lambda_m}$, and the convergence is uniform on $[a,b]^2$.
\end{proof}

\begin{rem}
We mention another way to obtain the Mercer's representation of $K$. A property of the Poincar\'e basis is that it is an orthogonal basis of $\Hone$ with $\Vert \varphi_m \Vert_{\Hone}^2 = 1 + \lambda_m = 1/\alpha_m$ \citep{SparsePoincareChaos}.
Thus, the functions $e_m = \sqrt{\alpha_m} \varphi_m$ ($m \geq 0$) define an orthonormal basis of $H^1(\mu)$.
Then, representation  \eqref{eq:kH1ortho} is obtained with the usual representation of a kernel in a separable RKHS \citep{berlinetRKHSbook}:
$$
    K(x,y) = \sum_{m \in \N} \alpha_m \varphi_m(x) \varphi_m(y) = \sum_{m \in \N} e_m(x) e_m(y).
$$
However, by this way, the convergence is a priori only pointwise, and it is not clear whether $K$ is continuous on $[a,b]^2$. Thus, another argument has been used here, coming from the Green's function point of view, to prove the kernel continuity.
\end{rem}

\begin{rem}
At first look, it may be surprising that the Neumann conditions $f'(a) = f'(b) = 0$ do not appear in the RKHS, whereas all basis functions $\varphi_m$ satisfy it (while being dense in $L^2(\mu)$). Actually, it is not difficult to see that any function $f$ of $H^1(\mu)$ can be approximated by a function $f_\epsilon \in H^1(\mu)$ that verifies the Neumann condition, simply by truncating $f$ on $[a+\epsilon, b - \epsilon]$ and extending it continuously by a constant on $[a, a+\epsilon]$ and $[b-\epsilon, b]$. As functions of $H^1(\mu)$ are continuous on the compact interval $[a, b]$ (still under our assumption on $\mu$), the approximation error can be made as small as wanted. 
\end{rem}

\subsection{Examples}
\label{sec:MercerH1examples}
\begin{example}[Case of the uniform distribution]
\label{example:uniform_kernel}
Let $\mu$ be the uniform distribution on 
$ [a, b]$. Then $L$ is the Laplacian operator, and the spectral problem is written
$$ \varphi''(x) = - \lambda \varphi(x) \qquad \forall x \in [a,b],$$
with Neumann conditions $\varphi'(a) = \varphi'(b)=0$.
The solutions are given by $\lambda_0 = 0, \varphi_0 = 1$
and for $m \geq 1$,
$$ \lambda_m = m^2 \omega^2, \qquad \varphi_m(x) = \sqrt{2} \cos(m \omega (x - a)), $$
with $\omega = \pi/(b-a)$.
The kernel of $H^1(\mu)$ (with its usual norm) has been obtained by \cite{these_duc_jacquet_1973} in the 70's. English-written proofs can be found in \citet{atteia}, Example 1.4, or \citet{thomasAgnan_H1_kernel}). The kernel is written
$$K(x,y) = \frac{b-a}{\sinh(b-a)}
\,\,
\cosh [\min(x, y) - a] 
\, 
\cosh [b - \max(x, y)] 
$$
where $\cosh, \sinh$ denote the hyperbolic functions: $\cosh(x) = \frac{e^x + e^{-x}}{2}$, $\sinh(x) = \frac{e^x - e^{-x}}{2}$. Applying Prop.\ref{prop:MercerH1}, we deduce that for all $(x, y) \in [a, b]^2$ such that $x \leq y$,
$$ K(x, y) = \frac{\pi / \omega}{\sinh(\pi / \omega)}
\,\,
\cosh(x - a)
\, 
\cosh(b - y)  = 1 +
2 \sum_{m=1}^{+\infty} \frac{1}{1 + m^2 \omega^2} \cos[m\omega(x-a)]
\cos[m\omega(y-a)].
$$
As a by-product, we can derive the value of some `shifted' Riemann series. For instance, from $x = y = a$ and $x=a, y=b$ and using $r=1/\omega$, we get the formulas (with $\tanh(x) = \sinh(x)/\cosh(x)$), valid for all $r>0$:
$$ \sum_{n=1}^{+\infty} \frac{1}{n^2 + r^2} = \frac{1}{2r^2} \left(\frac{ \pi r}{\tanh(\pi r)} - 1 \right), \qquad
\sum_{n=1}^{+\infty} \frac{(-1)^{n-1}}{n^2 + r^2}  = \frac{1}{2r^2} \left(1 - \frac{ \pi r}{\sinh(\pi r)} \right).$$
It can be shown that these formulas are also valid when $r$ tends to zero. The limit case gives the well-known expression at $s=2$ of the Riemann and Dirichlet eta functions: $\zeta(2)=\pi^2/6$ and $\eta(2) = \pi^2/12$.
\end{example}
\bigskip
In addition to the standard space $H^1(a,b)$ associated to the uniform distribution, the kernel of $H^1(\mu)$ and its Mercer's representation can also be made explicit in the case of the truncated exponential distribution.
\begin{example}[Truncated exponential distribution]
Consider the exponential distribution, truncated on $[a, b] \subseteq \R^+$: $\displaystyle \frac{d\mu}{dx} = \exp(-V(x)) = \frac{e^{-x}}{e^{-a}-e^{-b}} \mathbb{1}_{[a, b](x)}$. 
Notice that $V'(x) = 1$ on $[a, b]$, leading to linear differential equations with constant coefficients. 
Following \citet[Proof of Proposition 5]{PoincInterval}, 
the spectral problem
$$ \varphi'' - \varphi' = - \lambda \varphi $$
with Neumann conditions $\varphi'(a) = \varphi'(b) = 0$, admits the solutions, $\lambda_0=0, \varphi_0 \equiv 1$ and for $m \geq 1$,
\begin{equation} \label{eq:truncExpSpectr}
\lambda_m = \frac{1}{4} + (m \omega)^2, \qquad
\varphi_m(x) = c_m e^{x/2}(2 m \omega \cos(m \omega (x - a)) - \sin(m \omega (x-a)))
\end{equation}
where $\omega = \pi/(b-a)$ and $c_m$ is a normalizing constant ensuring that $\varphi_m$ has $L^2(\mu)$ norm $1$, equal to:
$$ c_m = \left( \frac{e^{-a} - e^{-b}}{b-a} \frac{1}{2 \lambda_m} \right)^{1/2}.$$
From Prop.~\ref{prop:MercerH1} and following \cite{oscillationBook}, the Green function associated to this spectral problem can be computed by considering the linear homogeneous equation
$$ \varphi'' - \varphi' - \varphi = 0.$$
The solutions are spanned by 
$e^{r_0 x}, e^{r_1 x}$, 
where $r_0 = \frac{1-\sqrt{5}}{2}$ 
and $r_1 = \frac{1+\sqrt{5}}{2}$.
Solutions $\psi, \chi$ satisfying one-sided Neumann conditions $\psi'(a) = 0$ and $\chi'(b) = 0$ are given by
\begin{equation}
\psi(x) = (r_1 e^{r_1 a}) e^{r_0 x} - (r_0 e^{r_0 a})e^{r_1 x}, 
\qquad
\chi(x) = (r_1 e^{r_1 b}) e^{r_0 x} - (r_0 e^{r_0 b})e^{r_1 x}.
\end{equation}
Finally, the normalization constant is given by $C = \int_a^b \psi(x) \chi(b) d\mu(x)$.  After some algebra, we obtain
$$ C = \frac{(e^{r_0 a + r_1 b} - e^{r_0 b + r_1 a})(r_1 - r_0)}{e^{-a} - e^{-b}}. $$
Finally the kernel of $H^1(\mu)$ is given explicitly by its single-pair form (\ref{eq:singlePairKernel}) 
or by its Mercer's representation (\ref{eq:H1kerMercer}), where $\alpha_m = (1+\lambda_m)^{-1}$ and $\lambda_m, \varphi_m$ are given by (\ref{eq:truncExpSpectr}).
\end{example}

\section{Poincaré quadrature and optimal kernel quadrature in $H^1(\mu)$}
\label{sec:PoincQuadPropGen}

Proposition~\ref{prop:MercerH1} 
shows that the spectral decomposition associated to Poincaré inequalities is in correspondence with the spectral decomposition of the kernel of $H^1(\mu)$, viewed as a RKHS. 
However, the kernel of $H^1(\mu)$ is in general not directly available, which makes optimal kernel quadrature in $\Hone$ intractable.
This motivates us to focus on the quadratures defined from the Poincaré basis, as eigenfunctions of $H^1(\mu)$.

\subsection{Definitions and notations}

\begin{defi}[Poincaré quadrature]
We call \emph{Poincaré quadrature} the Gaussian quadrature of the $T$-system of the Poincaré basis of $\mu$, as defined in Definition~\ref{def:GaussQuadTsystem}. We denote it $(X_P, w_P)$.
\end{defi}

We now establish a connection between the Poincaré quadrature and the kernel quadrature spanned by the Poincaré basis functions. To reach this goal, we first check that these kernel quadratures are properly defined, by showing that Assumption~\ref{state:GramInvert} is verified both for $K$ and its finite-dimensional approximation $K_M$, defined below. 

\begin{defi}[Finite-dimensional kernel for $H^1(\mu)$]
Let $M$ a non-zero integer, and consider the notations of Proposition~\ref{prop:MercerH1} and its assumptions. We set 
\begin{equation} \label{eq:K_M}
K_M = \sum_{m=0}^M \alpha_m \varphi_m \otimes \varphi_m
\end{equation}
the truncated Mercer's representation of the kernel of $H^1(\mu)$. \end{defi}

\begin{prop}[Invertibility of the Gram matrix for truncated Mercer's representation of $H^1(\mu)$] 
\label{prop:GramInvertH1trunc}
Assumption~\ref{state:GramInvert} is verified for $K_M$ when $X$ contains at most $M+1$ distinct points.
\end{prop}
\begin{proof}
Let $x_0, \dots, x_n$ a set of distinct points with $n \leq M$.
For $i=0, \dots, n$, denote $e_i = K_M(x_i, .) \in H^1(\mu)$. Then $K_M(x_i, x_j) = \langle e_i, e_j \rangle_{H^1(\mu)}$. Thus $K_M(X, X)$ is the Gram matrix of the vectors $e_i$ for the dot product in $H^1(\mu)$. By a classical result, it is invertible if and only if $e_0, e_1, \dots, e_n$ are linearly independent. It is enough to prove that $e_0, e_1, \dots, e_M$ are linearly independent, since $n \leq M$.
Now, by definition, we have
$$ e_i = K_M(x_i, .) = \sum_{m=0}^M A_{i,m} \varphi_m $$
with $A_{i, m} = \alpha_m \varphi_m(x_i)$. Hence, if $e = (e_1, \dots, e_M)^\top$ and $\varphi = (\varphi_0, \dots, \varphi_M)^\top$ are column vectors whose elements are in $H^1(\mu)$, then we have  $e = A \varphi$.
Remember that the $\varphi_m$'s are linearly independent ($m \geq 0)$, as they form a basis of $L^2(\mu)$. Thus the coordinates of $e$ are linearly independent if and only if $A$ is invertible. Remarking that the $m^{\text{th}}$ column of $A$ is proportional to the vector $(\varphi_m(x_i))_{0 \leq i \leq M}$ with a non-zero multiplicative coefficient $\alpha_m$. Hence, the rank of $A$ is equal to the rank of the matrix $(\varphi_{m}(x_i))_{0 \leq m, i \leq M}$. As the $\varphi_m$'s form a $T$-system, this matrix is invertible, which completes the proof.
\end{proof}

\begin{prop}[Invertibility and form of the Gram matrix for $H^1(\mu)$]
\label{prop:GramInvertH1}
Let $K$ be the kernel of $H^1(\mu)$, where $\mu$ is a probability distribution in $\bound$. Then,
Assumption \ref{state:GramInvert} is verified for all set $X$ composed of distinct knots. Furthermore, in that case, the precision matrix $K(X, X)^{-1}$ is a one-band matrix (or Jacobi matrix) of the form:
$$\begin{pmatrix}
a_1 & b_1 & 0 & 0 &\dots & 0 \\
b_1 & a_2 & b_2 & 0 &  & 0 \\
0 & b_2 & a_3 & b_3 & \ddots & \vdots \\
\vdots & \ddots & \ddots &\ddots &\ddots & 0 \\
0 &  & 0 &b_{n-2} & a_{n-1} & b_{n-1} \\
0 & 0 & \dots & 0 & b_{n-1} & a_n
\end{pmatrix}$$
\end{prop}
\begin{proof} 
By  Proposition \ref{prop:MercerH1} and Theorem 16 (page 238) of \cite{oscillationBook}, the single-pair kernel $K$ is oscillatory in the sense of definition 1 page 178 of the same reference, implying in particular that for all $X$ composed of distinct knots, $K(X, X)$ is invertible. The form of the precision matrix is derived in 
section II.3., example 6, pages 79-82 of the same reference.
\end{proof}

\subsection{Equivalence of Poincaré and optimal kernel quadratures in $H^1(\mu)$}

In the previous section, we proved that kernel quadrature is well defined for the kernel $K$ of $H^1(\mu)$ and its finite-dimensional approximation $K_M$. 
Now, we show that the Poincaré quadrature can be viewed as an optimal kernel quadrature with positive weights for the finite dimensional approximation of the kernel of $H^1(\mu)$ in its Mercer's representation.

\begin{prop}[Equivalence of Poincaré quadrature and optimal kernel quadrature in $H^1(\mu)$]
\label{prop:Poinc_and_kernel_quad}
Let $(X_P, w_P)$ be the Poincaré quadrature of $H^1(\mu)$ with $n$ nodes and order $M=2n-1$. Then, $\wce(X_P, w_P, K_M) = 0$ and $(X_P, w_P)$ is an optimal kernel quadrature for $\HkM$, with positive weights.\\
Conversely, if $(X, w)$ defines a kernel quadrature for $\HkM$ such that $\wce(X, w, K_M) = 0$ and the weights are positive, then $X=X_P$ and $w=w_P$.\\
Furthermore, $w_P$ minimizes over all (possibly negative) weights the worst-case error given $X_P$:
$$ w_P = w^\star(X_P, K_M) = \mathrm{argmin}_{w \in \R^n} \wce(X_P, w, K_M).$$
\end{prop}

\begin{rem} 
As a consequence of Prop.~\ref{prop:Poinc_and_kernel_quad}, the optimal weights $w^\star(X_P, K_M)$ are equal to $w_P$, and are thus positive and sum to $1$, which was not obvious a priori as they are defined by an optimization problem on $\R^n$. 
\end{rem}

\begin{proof} 
Let us first consider the Poincaré quadrature. We know that
\begin{equation} \label{eq:PoincQuadDef}
\int \varphi_m(x) d\mu(x) = \sum_{j=1}^n w_{j} \varphi_m(x^j), \qquad m=0, \dots, M.
\end{equation}
Let $x' \in [a,b]$. By linearity, we deduce from (\ref{eq:PoincQuadDef}) that
$$ \sum_{m=0}^M \alpha_m \varphi_m(x') \int \varphi_m(x) d\mu(x) = 
\sum_{m=0}^M \alpha_m \varphi_m(x')
\sum_{j=1}^n w_{j} \varphi_m(x^j),$$
or equivalently
$$ \int \sum_{m=0}^M \alpha_m \varphi_m(x') \varphi_m(x) d\mu(x) = \sum_{j=1}^n w_{j} \left(\sum_{m=0}^M \alpha_m \varphi_m(x')
 \varphi_m(x^j) \right),$$
i.e.
$$
\int K_M(x, x') d\mu(x) = \sum_{i = 1}^n w_i K(x^i, x').
$$
This proves that $wce(X_P, w_P, K_M) = 0$.\\
Conversely, let $(X, w)$ be a kernel quadrature for $\HkM$ such that $\wce(X_P, w_P, K_M) = 0$ and with positive weights. 
By definition of the worst-case error, this implies that for all $f$ in the unit ball of $\HkM$, $\int f(x) d\mu(x) = \sum_{i=1}^n w_i f(x_i)$.
By considering $f/\Vert f \Vert_{\HkM}$, this identity is true for all $f \in \HkM$. In particular for $f=\varphi_m$ with $m \leq M$, we deduce that the quadrature defined by $(X, w)$ is exact for all $\varphi_m$ such that $m \leq M$. Thus, by uniqueness of the Gaussian quadrature of the $T$-system $(\varphi_m)_{m \in \N}$ (Prop.~\ref{prop:GaussQuadTsystem}), we deduce that $X = X_P$ and $w=w_P$.\\
Now, denote $X_P = \{x_1, \dots, x_n \}$. Recall that the nodes $x_i$ are all different by a property of the (generalized) Gaussian quadrature.
By Proposition~\ref{prop:GramInvertH1trunc}, the Gram matrix $(K_M(x^i, x^j))_{0 \leq i, j \leq n}$ is then invertible. 
This implies that the minimization problem 
$$ \min_w \wce(X_P, w, K_M)$$
has a unique solution $w^\star(X_P, K_M)$. Since $\wce(X_P, w_P, K_M) = 0$ and $\wce(X_P, w, K_M) \geq 0$ for all $w$, we obtain $w^\star(X_P, K_M) = w_P$, which concludes the proof.
\end{proof}

\subsection{Formulas for optimal weights}
Exploiting the equivalence of Poincaré quadrature and kernel quadrature, we obtain several formulas for optimal weights.

\begin{prop}[Expression of the optimal weights and associated worst-case error for $H^1(\mu)$] 
\label{prop:weights_wce_H1mu}
Let $K$ be the kernel of $H^1(\mu)$. Then, for all set $X$ composed of distinct quadrature knots, we have
\begin{equation} \label{eq:optimalWeightH1}
w^\star(X, K) = K(X, X)^{-1} \mathbb{1}, 
\end{equation}
where $\mathbb{1}$ is the vector of ones of length $n$,
and 
\begin{equation}
\label{eq:wceOptimalH1}
 (\wce(X, w^\star(X, K), K))^2 = 1 - \mathbb{1}^\top K(X, X)^{-1} \mathbb{1} = 1 - \sum_{i=1}^n w_i^\star(X, K).
\end{equation}
Similarly, when $X$ is formed by at most $M+1$ distinct points, Equations (\ref{eq:optimalWeightH1}) and (\ref{eq:wceOptimalH1}) are true when replacing $K$ by $K_M$. 
In particular, 
\begin{equation} \label{eq:wP_KM}
    w_P = K_M(X_P, X_P)^{-1} \mathbb{1}.
\end{equation}
\end{prop}
\begin{proof} 
First recall that Assumption~\ref{state:GramInvert} is verified both for $K$ and $K_M $ (Prop.~\ref{prop:GramInvertH1} and \ref{prop:GramInvertH1trunc}), in the latter case when $X$ has at most $M+1$ points.
Let us first consider the case of the kernel of $H^1(\mu)$. Consider the Mercer representation of $K$, 
$$ K(x,y) = \sum_{m=0}^\infty \alpha_m \varphi_m(x) \varphi_m(y).$$
Recall that, as $K$ is continuous, the convergence is uniform on the compact set $[a, b]^2$. Furthermore, the $\varphi_m$ are also continuous on $[a, b]$. Thus, for all $y$ in $[a,b]$, we have
$$ \int_a^b K(x,y) d\mu(x) = \sum_{m=0}^\infty \alpha_m \varphi_m(y) \int_a^b \varphi_m(x) d\mu(x).$$
Now, as $\int_a^b \varphi_m d\mu = \delta_{0, m}$. Thus, 
$$ \ell_K(y) =  \int_a^b K(x, y) d\mu(x) = 1.$$
From (\ref{eq:optimalWeight}), we then deduce (\ref{eq:optimalWeightH1}). Finally, from (\ref{eq:wceOptimal}), we deduce (\ref{eq:wceOptimalH1}), using that $c_K = \int_a^b \ell_K(y) d\mu(y) = 1$.\\
The same proof applies when $K=K_M$, replacing the full Mercer's representation by a partial sum. In that case, we recover the fact that $wce(X_P, w_P, K_M) = 0$, since the weights $w_P$ sum to one. From Proposition~\ref{prop:Poinc_and_kernel_quad}, we have $w_P = w^\star(X_P, K_M)$. As $X_P$ contains $n \leq M+1$ distinct points, we deduce (\ref{eq:wce_H1_2}).\\
Finally, the fact that the precision matrix $K(X, X)^{-1}$ is one-band has been proved in Prop. \ref{prop:GramInvertH1}. 
\end{proof}

\subsection{Quadrature error}
\label{sec:quadError}
The quadrature error can be quantified using the \emph{radius of information}, which is defined as the smallest worst-case error of the optimal kernel quadrature with $n$ nodes:
\begin{equation} \label{eq:radius}
r(n) = \inf_{X, w} \wce(X, w, \Hone).
\end{equation}
In our case, the radius of information can hardly be computed.
However, we can consider the Poincaré quadrature, which is the optimal kernel quadrature with positive weights of the finite-dimensional approximation of $H^1(\mu)$, and compute the corresponding worst-case error in $\Hone$. More precisely, if $(X_P, w_P)$ denotes the Poincaré quadrature with $n$ nodes and order $M=2n-1$, we set:
\begin{equation} \label{eq:wce}
\wce(n) = \wce(X_P, w_P, \Hone).
\end{equation}
By definition we have $\wce(n) \geq r(n)$. Furthemore, when $n$ is large, $K_M$ tends to $K$ and we can hope that $\wce(n)$ is a good approximation of $r(n)$. We now provide formulas for $\wce(n)$.

\begin{prop}[Quadrature error]
\label{prop:wceH1}
The worst-case error of the Poincaré quadrature with $n$ nodes and order $M=2n-1$ can be expressed with the Mercer's representation of $H^1(\mu)$, by:
\begin{equation}
\wce(n)^2 = \sum_{m \geq M+1} \alpha_m 
\left(\sum_{i=1}^n w_i \varphi_m(x_i)\right)^2,
\label{eq:wce_spectral}
\end{equation}
or with formulas involving the kernel of $H^1(\mu)$:
\begin{eqnarray}
\wce(n)^2 &=& w_P^\top (K(X_P, X_P) - K_M(X_P, X_P)) w_P \label{eq:wce_H1_1} \\
&=& \mathbb{1}^\top K_M(X_P, X_P)^{-1} 
(K(X_P, X_P) - K_M(X_P, X_P)) K_M(X_P, X_P)^{-1} \mathbb{1} \label{eq:wce_H1_2} 
\end{eqnarray}
Furthermore, 
we have, for all $n \in \N$,
$$ \wce(n) \leq \sqrt{\Vert K - K_{2n-1} \Vert_{\infty}}, $$
which goes to zero when $n$ tends to infinity. 
\end{prop}

\begin{proof}[Proof of Proposition \ref{prop:wceH1}] Denote by $L$ the linear form on $\mathcal{H}_K$ defined by
$$L(f; X, w) = \int f(x) d\mu(x) - \sum_{i=1}^n w_i f(x_i)$$
and $$R_M(x, x') = \sum_{m \geq M+1} \alpha_m \varphi_m(x) \varphi_m(x').$$
Notice that $ K = K_M + R_M.$
For the Poincaré quadrature, we have:
$$ L(K_M(., x'); X_P, w_P) =  \sum_{m = 0}^M \alpha_m \varphi_m(x')
\left( \int \varphi_m(x) d\mu(x) 
- 
\sum_{i=1}^n w_i \varphi_m(x_i)
\right) = 0$$
as the quadrature is exact for the eigenfunctions up to order $M$.
Hence, by linearity, it holds for all $x'$: 
$$L(K(., x'); X_P, w_P) =  L(R_M(., x'); X_P, w_P).$$
Now, remark that for all quadrature formula $(X, w)$, 
$$ \wce(X, w, K) = \Vert x' \mapsto L(K(., x')) \Vert_{\mathcal{H}_K}.$$ Therefore, we have $\wce(X_P, w_P, K_M) = 0$ and  $\wce(X_P, w_P, K) = \wce(X_P, w_P, R_M) $. Thus, the quantity of interest reduces to:
$$\wce(n) := \left \vert \wce(X_P, w_P, K_M) - \wce(X_P, w_P, K) \right \vert = \wce(X_P, w_P, R_M).$$
Now, when $m \geq M+1$, we have $\int \varphi_m d\mu=0$. Thus,
\begin{eqnarray*}
\wce(n)^2 &=& 
\Vert x' \mapsto L(R_M(., x'); X_P, w_P) \Vert_{\mathcal{H}_K}^2 \\
&=& \left\Vert \sum_{m \geq M+1}\alpha_m \left(\sum_{i=1}^n w_i \varphi_m(x_i)\right) \varphi_m  \right\Vert_{\mathcal{H}_K}^2
\end{eqnarray*}
As $(\varphi_m)$ is an orthogonal basis of $H^1(\mu)$ with $\Vert \varphi_m \Vert_{\mathcal{H}_K}^2 
= 1+\lambda_m=\alpha_m^{-1}$, we immediately obtain
$$ \wce(n)^2 = \sum_{m \geq M+1} \alpha_m 
\left(\sum_{i=1}^n w_i \varphi_m(x_i)\right)^2.$$
Remarking that $R_M(x_i, .) = \displaystyle \sum_{m \geq M+1} \alpha_m \varphi_m(x_i) \varphi_m(.)$, we get
$$ \wce(n)^2 = \left\Vert 
\sum_{i=1}^n w_i R_M(x_i, .) \right\Vert_{\mathcal{H}_K}^2$$
which gives, using the reproducing property (as $R_M(x_i, .) \in \Hk$),
$$ \wce(n)^2 = \sum_{1 \leq i, j \leq n} w_i w_j R_M(x_i, x_j)$$
which gives (\ref{eq:wce_H1_1}).
Using (\ref{eq:wP_KM}), we deduce (\ref{eq:wce_H1_2}).
By Mercer's theorem, as $K$ is continuous, the series $\sum_m \alpha_m \varphi_m(x) \varphi_m(y)$ converges to $K(x, y)$ uniformly on $[a, b]\times [a,b]$. 
Thus $R_M$ goes to zero uniformly on $[a, b]^2$, and, using the positivity of the weights, 
$$ \wce(n)^2 \leq \left(\sum_{1 \leq i,j \leq n} w_i w_j \right) \Vert R_M \Vert_\infty.$$
The results follows by remarking that 
$\sum_{1 \leq i,j \leq n} w_i w_j = 
\left( \sum_{i=1}^n w_i \right)^2 = 1.$
\end{proof}

\section{The case of $H^1(a,b)$}
\label{sec:PoincQuadPropUnif}

\subsection{Nodes coincide with zeros of a basis function} 

For polynomials, the nodes of the Gaussian quadrature coincide with the zeros of an orthogonal polynomial. The main result of this section can be viewed as an extension of this property to certain $T$-systems. The key argument is that, in the case of the uniform distribution, the quadrature is not only exact for the functions of the $T$-system (up to some order) but also for their products. This guarantees that the quadrature nodes coincide with the zeros of an element of the $T$-system.

\begin{lem}[Stability under multiplication] \label{lem:PoincareUnifOrtho}
Let $\mu$ be the uniform distribution on $[a,b]$. 
Let $n, j, k \in \N$ such that $j<n$ and $k \leq n$. Then, the Poincaré quadrature with $n$ nodes is exact for the product of eigenfunctions $\varphi_j \varphi_k$ and for the product of their derivatives at any order $\varphi_j^{(\ell)} \varphi_k^{(\ell)} (\ell \geq 1)$. In particular, if $j \neq k$, it preserves the orthogonality of $\varphi_j^{(\ell)}$ and $\varphi_k^{(\ell)}$ for all $\ell \in \N$.
\end{lem}

\begin{proof}
Recall that when $\mu$ is the uniform distribution on $[a,b]$, we have $\varphi_m(x) 
=\sqrt{2} \cos(m \omega (x - a))$ 
with $\omega = \pi/(b-a)$.
Using the trigonometric identity,  
$$ \cos(u) \cos(v) = \frac{1}{2} \left( \cos(u+v) + \cos(u-v) \right)$$
it follows that 
$$ \varphi_n(x) \varphi_m(x) 
= \frac12 \left( \varphi_{n+m}(x) + \varphi_{n-m}(x)\right).
$$
Now, by Definition~\ref{def:GaussQuadTsystem}, the Gaussian quadrature of the $T$-system $(\varphi_m)_{m \in \N}$ with $n$ nodes is exact for $\varphi_m$ for $0 \leq m \leq 2n-1$.
Let $j,k \in \N$ such that $j <n$, $k \leq n$.  
Without loss of generality, assume $k \geq j$.
This implies that both $k+j$ and $k-j$ are non-negative and lower or equal than $2n-1$. Thus, the quadrature is exact for $\varphi_{k+j}$ and $\varphi_{k-j}$ and we have:
\begin{align*}
    \delta_{kj} = \int \varphi_k(x) \varphi_j(x) d\mu(x) 
    &= \frac12 \int \varphi_{k+j}(x) d\mu(x) + \frac12 \int \varphi_{k-j}(x) d\mu(x) \\
    &= \frac12 \sum_{i=1}^n w_i \varphi_{k+j}(x_i) + \frac12 \sum_{i=1}^n w_i \varphi_{k-j}(x_i) \\
    &= \sum_{i=1}^n w_i \frac12 \left( \varphi_{k+j}(x_i) + \varphi_{k-j}(x_i) \right) \\
    &= \sum_{i=1}^n w_i \varphi_k(x_i) \varphi_j(x_i).
\end{align*}
To prove the result for the derivatives, it is enough to consider the first-order derivatives, as the higher order derivatives of cosine is either proportional to the function itself or to its first derivative. Now, by a property of the Poincaré basis,
$$\langle \varphi'_j, \varphi'_k \rangle = \lambda_j  \langle \varphi_j, \varphi_k \rangle = \lambda_j \delta_{k,j}.$$
Then, using the trigonometric identity $\sin(u)\sin(v) = \frac{1}{2} (\cos(u-v) - \cos(u+v))$, one can check that 
$ \varphi'_j \varphi'_k = A (\varphi_{j-k} - \varphi_{j+k})$
for some constant $A$. Then, in a similar way as above, we have: 
\begin{eqnarray*}
\sum_{i=1}^n \omega_i \varphi_j'(x_i) \varphi_k'(x_i) 
&=& \sum_{i=1}^n \omega_i A (\varphi_{j-k}(x_i) - \varphi_{j+k}(x_i)) \\
&=& A \int \varphi_{j-k}(x)d\mu(x) - A \int \varphi_{j+k}(x) d\mu(x) \\
&=& \int \varphi'_j(x) \varphi'_k(x) d\mu(x) =  \lambda_j \delta_{j,k}.
\end{eqnarray*}
\end{proof}

\begin{prop} \label{prop:gaussQuadZeros}
Let $\mu$ be a probability distribution in $\bound$. 
Let $u=(u_m)_{m \in \N}$ be a $T$-system formed by orthogonal functions in $L^2(\mu)$, such that
for all $j,k \in \N$ such that $j < n$, $k \leq n$, the Gaussian quadrature of $u$ with $n$ nodes is exact for the product of functions $u_j u_k$.
Then the quadrature nodes $x_1, \dots, x_n$ coincide with the zeros of $u_n$.
\end{prop}

\begin{proof}
Following \citep[p.~20]{karlin1966t}, as $u$ is a $T$-system, there exists a generalized polynomial $p_n = \alpha_0 u_0 + \ldots + \alpha_n u_n$ that vanishes at $x_1, \dots, x_n$. 
It can be defined from the generalized Vandermonde matrix by:
$$ p_n(x) 
= \det V(u_0, \dots, u_{n-1}, \varphi_n; x_1, \dots, x_n, x).$$
Let $j \in \N$ with $j < n$. By hypothesis, for all $k \in \N$ such that $k \leq n$, the Gaussian quadrature of $u$ is exact for $u_j u_k$. 
This implies that $p_n$ is orthogonal to $u_j$:
\begin{align*}
    \int p_n(x) u_j(x) d\mu(x) &= \sum_{k=0}^n \alpha_k \int u_k(x) u_j(x) d\mu(x) \\
    &= \sum_{k=0}^n \alpha_k \sum_{i=1}^n w_i u_k(x_i) u_j(x_i) \\
    &= \sum_{i=1}^n w_i  \underbrace{\sum_{k=0}^n \alpha_k u_k(x_i)}_{= p_n(x_i) = 0} u_j(x_i) = 0.
\end{align*}
By orthogonality of the $u_i$'s, we deduce that $p_n$ is proportional to $u_n$. Hence, $u_n$ is zero at the quadrature nodes, which was to prove.
\end{proof}

Applying Prop.~\ref{prop:gaussQuadZeros} to the $T$-system of orthogonal polynomials, we recover the well-known link between the zeros and the nodes for the Gaussian quadrature of polynomials (\cite{Szego1959}). In that case indeed, the exactness of the quadrature for $u_ju_k$
is a consequence of the stability by multiplication of polynomials, as $u_ju_k$ is a polynomial of degree less than $2n-1$.
Coming back to the Poincaré basis, we immediately deduce from Lemma~\ref{lem:PoincareUnifOrtho} and Prop.~\ref{prop:gaussQuadZeros} the announced result of the section:

\begin{cor} \label{prop:gaussQuadUnifCase}
The nodes of the Poincaré quadrature of $H^1(a,b)$ with $n$ nodes are equal to the zeros of the Poincaré basis function $\varphi_n$.
\end{cor}

\subsection{Explicit quadrature formulas and quadrature error}

\begin{lem} 
\label{lem:trigoSum}
For all $m \in \mathbb{Z}$ and all $n \in \N^\star$,
$$\sum_{i=1}^{n} \cos \left(
\left(i-\frac{1}{2} \right) \frac{m \pi}{n} \right) = 
\begin{cases}
0 & \textrm{if $m$ is not a multiple of $2n$}\\
n (-1)^p & \textrm{if $m=(2n)p$, for all $p \in \mathbb{Z}$}.
\end{cases}
$$
\end{lem}

\begin{proof}
The proof is standard in computing trigonometric sums.
Let $x = (m \pi) / n$. \\
If $m$ is a multiple of $2n$, then $x=2p\pi$ for some $p \in \mathbb{Z}$. Then 
$$\sum_{i=1}^{n} \cos \left(
\left(i-\frac{1}{2} \right) x \right)
= \sum_{i=1}^n \cos(-p\pi) = \sum_{i=1}^n (-1)^p = n (-1)^p.
$$
Now, assume that $m$ is not a multiple of $2n$. Then $x/2$ is not a multiple of $\pi$ and $\sin(x/2) \neq 0$. Using the trigonometric identity
$$ 2\cos a \sin b =  \sin(a+b) - \sin(a-b),$$
we have
$$ 2 \cos \left(
\left(i-\frac{1}{2} \right) x \right) \sin \left( \frac{x}{2} \right) =
\sin(ix) - \sin((i-1)x).$$
Summing with respect to $i$ then gives, by telescoping, 
$$ 
2\sin \left( \frac{x}{2} \right)
\times 
\sum_{i=1}^{n} \cos \left(
\left(i-\frac{1}{2} \right) x \right)  = \sin( nx) = \sin(m\pi) = 0.
$$
As we are in the case where $\sin(x/2)\neq 0$, this concludes the proof.
\end{proof}

\begin{prop}
The Poincaré quadrature of $H^1(a,b)$ with $n$ nodes corresponds to the midpoint (or rectangle) quadrature rule
$$ \int_a^b f(x) dx  = \frac{b-a}{n} \sum_{i=1}^n  f \left( 
a + \left(i-\frac{1}{2} \right) \frac{b-a}{n}  \right) . $$
Thus, the optimal weights are equal to $1/n$, and the optimal nodes are evenly spaced on $[a, b]$ and located at the middle of each interval $\displaystyle \left[a + (i-1)\frac{b-a}{n}, a+i\frac{b-a}{n}
\right]$, $i=1, \dots, n$.
This quadrature has order $2n-1$ with respect to the generalized polynomials: it is exact for all $\varphi_m \propto \cos \left( m \pi \frac{x-a}{b-a} \right)$ with $m \leq 2n-1$. Furthermore it is also exact for all $\varphi_m$ such that $m$ is not a multiple of $2n$, and for polynomials of order 1.\\
\end{prop}

\begin{proof}
From the previous section, the nodes of the Poincaré quadrature of $H^1(a,b)$ coincide to the zeros of $\varphi_n = \sqrt{2} \cos \left( n \pi \frac{x-a}{b-a} \right)$ on $[a, b]$. Hence they are equal to $x_i = a + \left(i-\frac{1}{2} \right) \frac{b-a}{n}$, for $i=1, \dots, n$.\\
Now, 
$\varphi_m(x_i) 
=\sqrt{2} \cos(m \omega (x_i - a)) 
= \sqrt{2} \cos \left(
\left( i - \frac{1}{2} \right) \frac{m \pi}{n}
\right).$
Thus, by Lemma \ref{lem:trigoSum}, 
$\sum_{i=1}^{n} \varphi_m(x_i) = 
0$ if $m$ is not a multiple of $2n$.
Hence, if we set $w_i=\frac{1}{n}$ and if $m$ is not a multiple of $2n$, then
$$ \int_a^b \varphi_m(x) \frac{dx}{b-a} = \delta_{m, 0} = \sum_{i=1}^{n} w_i \varphi_m(x_i),$$
where the case $m=0$ is equivalent to $w_1 + \dots + w_n = 1$.
Recall that the quadrature weights are uniquely determined by the first $n$ equations above, corresponding to $m=0, \dots, n-1$. Indeed, the matrix of the linear system is $(\varphi_m(x_i))_{0 \leq m,i \leq n-1}$, which is invertible by the T-system property.
This proves that the optimal weights are equal to $1/n$. Furthermore, the same equations show that the quadrature rule is exact for all $\varphi_m$ such that $m$ is not a multiple of $2n$.
Finally, the quadrature is interpreted as the midpoint quadrature rule, which is exact for all polynomials of order $1$.
\end{proof}

\begin{prop}[Quadrature error]
\label{prop:quadError}
Consider the quadrature error defined as the worst-case error of the Poincaré quadrature of $H^1(a, b)$ with $n$ nodes (see Section~\ref{sec:quadError}). We have:
\begin{equation}
\wce(n) = 
\left(\frac{\frac{b-a}{2n}}{\tanh \left(\frac{b-a}{2n} \right)} - 1\right)^{1/2}
\end{equation}
and goes to zero at a linear speed when $n$ tends to infinity:
$$ \wce(n) \sim  \frac{b-a}{2\sqrt{3}}\frac{1}{n}.$$
\end{prop}

\begin{proof}
Recall that the optimal weights are $w_i= 1/n$, and the optimal nodes are $x_i = a + \left(i - \frac{1}{2} \right) \frac{b-a}{n}$.
Then, using (\ref{eq:wce_spectral}), we have:
$$ \wce(n)^2 = \sum_{m=2n}^{\infty} \alpha_m \frac{1}{n^2} \left( \sum_{i=1}^n \varphi_m(x_i) \right)^2.$$
Now, 
$\varphi_m(x_i) 
=\sqrt{2} \cos(m \omega (x_i - a)) 
= \sqrt{2} \cos \left(
\left( i - \frac{1}{2} \right) \frac{m \pi}{n}
\right),$
with $\omega = \frac{\pi}{b-a}$. By Lemma \ref{lem:trigoSum}, 
$$\sum_{i=1}^{n} \varphi_m(x_i) = 
\begin{cases}
0 & \textrm{if $m$ is not a multiple of $2n$}\\
\sqrt{2} n (-1)^p & \textrm{if $m=(2n)p$, for all positive integer $p$}.
\end{cases}
$$
Thus in this sum above, the non-zeros terms are such that $m=2pn$. Observe that $m \geq 2n$ is then equivalent to $p \geq 1$. Hence, reparameterizing by $p$, we obtain
$$\wce(n)^2 = 2 \sum_{p = 1}^\infty \alpha_{2np} = 2 \sum_{p=1}^\infty \frac{1}{1+p^2 / r^2}.
$$
with $r=(2n\omega)^{-1}$. Following the computations of Example 1, we have 
$$ \wce(n)^2 = 2 \sum_{p=1}^\infty \frac{1}{1 + p ^2 / r^2} 
= 2 r^2 \sum_{p=1}^\infty \frac{1}{p^2 + r^2} 
=  
\frac{\pi r}{\tanh(\pi r)} - 1.$$
This gives the explicit form of $\wce(n)$. To obtain the speed of convergence, notice that by an immediate application of Lebesgue theorem,  
$\sum_{p=1}^\infty \frac{1}{p^2 + r^2}$ tends to $\zeta(2)=\frac{\pi^2}{6}$ when $r$ tends to zero. Hence, when $n$ tends to infinity,
$$\wce(n)^2 \sim  \frac{1}{2n^2} \frac{1}{\omega^2} \frac{\pi^2}{6} = \frac{(b-a)^2}{12} \frac{1}{n^2}.$$
\end{proof}

\subsection{Asymptotical optimality of the Poincaré quadrature}
In the particular case of the uniform distribution, the kernel of $H^1(a, b)$ is given explicitly (see Section~\ref{sec:MercerH1examples}). Thus it is possible to derive the optimal kernel quadrature for $H^1(a, b)$, which has been done in \cite{these_duc_jacquet_1973}. For $a=0, b=1$, it is proved that the optimal nodes are $x_i^\star = \frac{2i-1}{2n}$, thus corresponding to the nodes of the rectangle quadrature, and the optimal weights are 
$w_i^\star = 2 \tanh \left( \frac{1}{2n} \right).$
For large $n$, $w_i^\star \sim \frac{1}{n}$. Thus, the Poincaré quadrature, here equal to the rectangle quadrature, is asymptotically equivalent to the optimal kernel quadrature. Furthermore, the radius of information (worst-case error for the optimal quadrature, see Section~\ref{sec:quadError}) verifies
$$ r(n)^2 \sim \frac{1}{12n^2}.$$
This is the same convergence speed as the worst-case error of the Poincaré quadrature, which we derived in Prop.~\ref{prop:quadError}:
$$ \frac{\wce(n)}{r(n)} \underset{n \to \infty}{\rightarrow} 1.$$
Therefore, we can conclude that the Poincaré quadrature is asymptotically optimal for $H^1(0, 1)$. This result is intuitive as the finite-dimensional kernel $K_M$, for which the Poincaré quadrature is optimal, tends to the kernel of $H^1(\mu)$.

\section{Numerical experiments}
\label{sec:NumExp}
\subsection{Numerical computation of the Poincar\'e quadrature}

\paragraph{Computation of the spectral decomposition.}
The first step to compute numerically the Poincaré quadrature is to obtain the spectral decomposition of Theorem~\ref{thm:spectralThm}.
This was investigated by \citet[][Section 4.2.]{PoincInterval}, who proposed a finite element technique. It consists of solving Problem (P2) in the finite-dimensional space spanned by $N$ piecewise linear functions whose knots are evenly spaced, which then boils down to a matrix diagonalisation problem. 
The theory of finite elements quantifies the speed of convergence  when $N$ tends to infinity, depending on the regularity of the probability density function $\rho$ of $\mu \in \bound$. If $\rho$ is of class $C^{\ell}$, the Poincar\'e basis functions are of class $C^{\ell+1}$, and the order of convergence is $O(N^{-2\ell})$ for the eigenvalues and $O(N^{-\ell})$ for the eigenfunctions.  
Notice that the value of $N$, controlling the mesh size, should be must larger than the order of the eigenvalue (or eigenfunction) to estimate. In practice we choose $N = 1\,000$.
Figure~\ref{fig:PoincareBasis} illustrates the result for the uniform distribution on $(0,1)$ and the exponential distribution truncated on $(0,3)$, for which the spectral decomposition is known in closed-form (as detailed in Section~\ref{sec:MercerH1examples}). We can see that the numerical approximation is accurate.

\begin{figure}[h!]
    \centering
    \includegraphics[width=0.8\linewidth, trim=0cm 1cm 0cm 0.5cm, clip=TRUE]{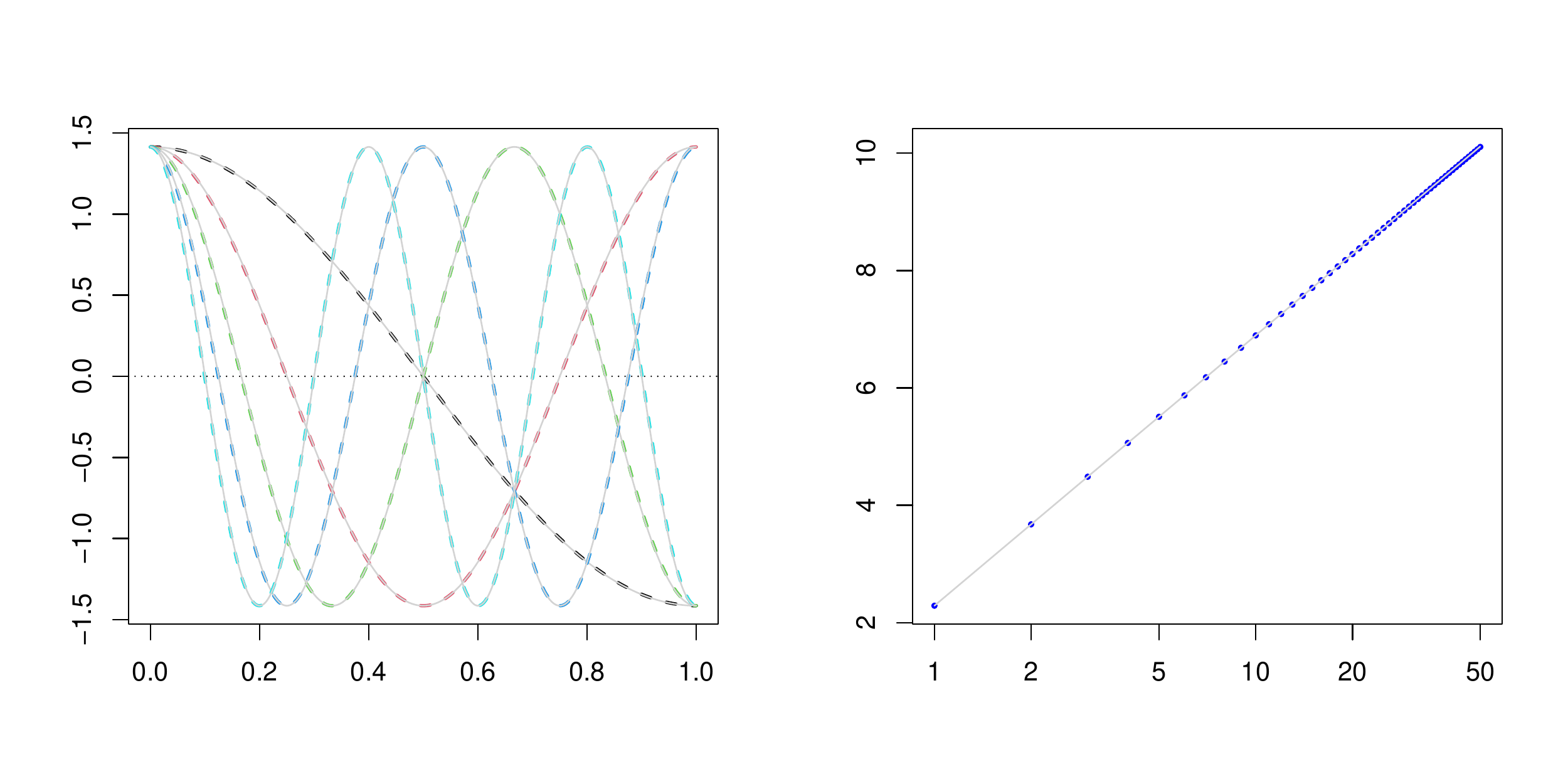}
    \includegraphics[width=0.8\linewidth, trim=0cm 1cm 0cm 1.5cm, clip=TRUE]{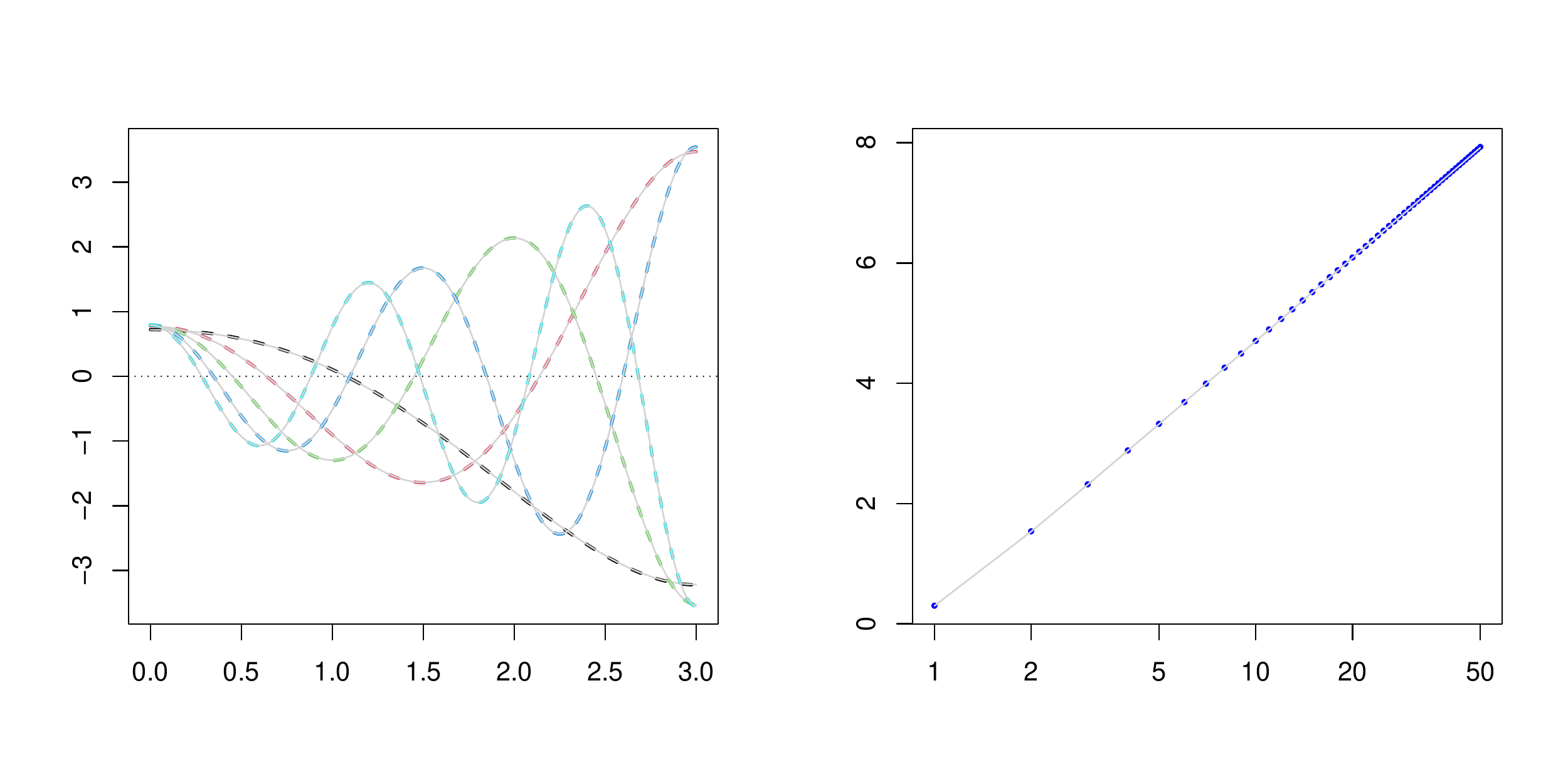}
    \caption{Poincar\'e spectral decomposition for the uniform distribution on $(0,1)$ (top) and the exponential distribution $d\mu(x) = e^{-x} 1_{\R^+}(x)$, truncated on $(0,3)$ (bottom).
    The left panel represent the first five eigenfunctions (Poincar\'e basis) and the right panel the first 50 eigenvalues in log scale. The colored dotted lines are the approximations computed by finite elements, and the (superimposed) grey solid lines are the true expressions.}
    \label{fig:PoincareBasis}
\end{figure}

\paragraph{Computation of the Poincar\'e quadrature.}
We now assume that the Poincar\'e basis has been computed numerically, as explained in the previous paragraph.
We aim at computing the Poincar\'e quadrature with $n$ nodes.
By definition, the Poincar\'e quadrature is the (generalized) Gaussian quadrature of the Poincar\'e basis. Thus, it can be obtained by solving the minimization problem~(\ref{eq:RepRec}) of Proposition~\ref{prop:GaussQuadTsystem} over probability distributions $\sigma$ subject to moment conditions. More precisely, inspired by the work of \cite{Gauss_quad_LP_Ryu_Boyd_2015} for the usual Gaussian quadrature, we search $\sigma$ as a discrete mesure supported on a fine uniform grid. 
We thus choose a large integer $N \gg n$ and consider the grid formed by evenly spaced points $z_j = a+jh$ ($j=0, \dots, N-1)$ where $h = \frac{b-a}{N-1}$ is the grid size. Searching for $\sigma$ of the form $\sigma=\sum_{j=0}^{N-1} w_j \delta_{z_j}$, 
(\ref{eq:RepRec}) is then rewritten as the following linear programming (LP) problem:
\begin{align} \label{eq:RepRecDiscrete}
    w^* = \arg{\min_{w \in [0, 1]^N}} \sum_{j=0}^{N-1} w_j \varphi_{2n}&(z_j) \\
    \nonumber
    \text{subject to} \quad & \sum_{j=0}^{N-1} w_j = 1 \\ \nonumber
    \text{and} \quad &  \sum_{j=0}^{N-1} w_j \varphi_i(z_j) = \int_a^b \varphi_i d\mu = \delta_{i, 0} \qquad (i = 0, \ldots, 2n-1).
\end{align}
The problem can be solved numerically by standard LP solvers. 
However, as the grid points may not contain exactly the unknown nodes, the solution $\sigma^\star = \sum_{j=0}^{N-1} w_j^\star \delta_{z_j}$ is generally supported on more than $n$ points. Thus, as a postprocessing step, we follow \cite{Gauss_quad_LP_Ryu_Boyd_2015} and apply a clustering technique with $n$ clusters (typically the $k$-means algorithm) to approximate the support points of the distribution. In each cluster $\mathcal{C}_i$ ($i=1, \dots, n$), a node $x_i$ is defined as a convex combination of its elements $z_j$ with weights proportional to $w^\star_j$ ($j \in \mathcal{C}_i$).
Finally, the weight $w_i$ associated to this node is defined as the sum of the weights in $\mathcal{C}_i$.

Due to the finite grid and the heuristic clustering and averaging technique, the moment conditions are in general not fulfilled exactly.
To improve the accuracy of the obtained solution $(x,w)$ further, we include a second optimization step, again following \citet{Gauss_quad_LP_Ryu_Boyd_2015}, in which we minimize the sum of squared moment conditions
\begin{align}
    (X^P, w^P) = \arg \min_{x, w} \sum_{i = 0}^{2n-1} & \left( \delta_{i,0} - \sum_{j=1}^{n} w_j \varphi_i(x_j) \right)^2 \label{eq:second_opt}\\
    \text{subject to} \quad & a \leq x_j \leq b \qquad (j = 1, \ldots, n) \nonumber\\
    \text{and} \quad & 0 \leq w_j \leq 1 \qquad (j = 1, \ldots, n) \nonumber
\end{align}
over both $x$ and $w$ using the interior-point algorithm (since we compute lower principal representations, which are always in $(a,b)$). The nodes and weights obtained from postprocessing the solution to \eqref{eq:RepRecDiscrete} are used as the starting point for \eqref{eq:second_opt}.
The objective function value of the solution $(X^P, w^P)$ of \eqref{eq:second_opt} is usually in the order of $10^{-7}$. For the uniform distribution, it is in the order of $10^{-16}$.

To evaluate the accuracy of the numerical quadrature, we compute the Poincar\'e quadrature of the uniform distribution where the analytical result is known (see Section~\ref{sec:PoincQuadPropUnif}). We have used the exact expression of the Poincar\'e basis.
Figure~\ref{fig:PoincQuadUnif} shows the results for $n=3$ and  $n=5$ nodes. We can see that the nodes found coincide with the zeros of $\varphi_n$, as expected by the theory. Furthermore, the weights are equal to $1/n$ (up to machine precision), which is also expected.

\begin{figure}[htbp]
\label{fig:PoincQuadUnif}
    \centering
    {\includegraphics[width=.49\textwidth, trim=0cm 0cm 0cm 0.8cm, clip=TRUE]{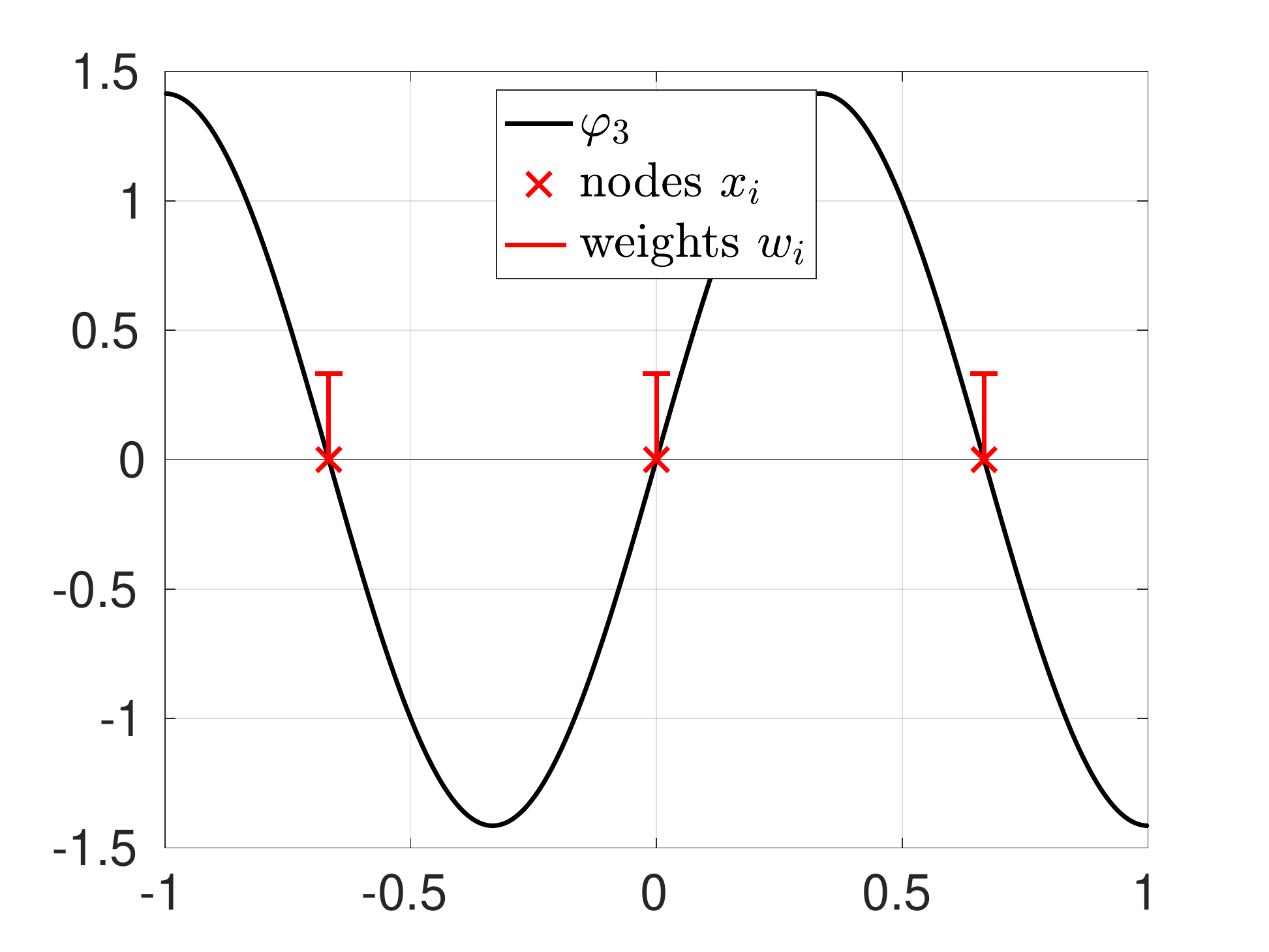}}
    \hfill
    {\includegraphics[width=.49\textwidth, trim=0cm 0cm 0cm 0.8cm, clip=TRUE]{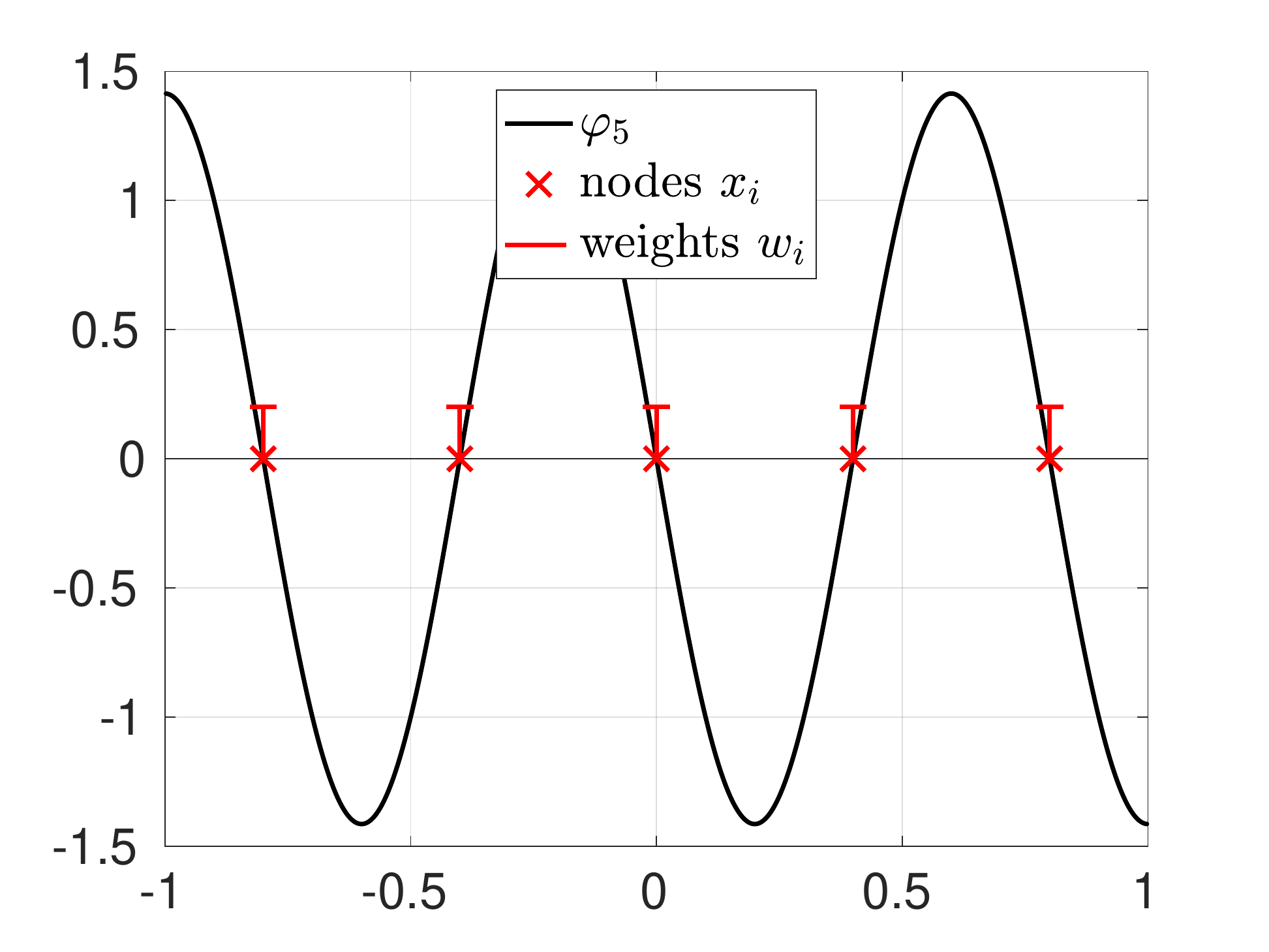}}
    \caption{Poincar\'e quadrature for the uniform distribution on $(0,1)$, with $n=3$ nodes (left) and $n=5$ nodes (right). The curve represents the Poincar\'e basis function with $n$ roots, and the red crosses and lines the quadrature nodes and weights obtained by the numerical procedure.}
\end{figure}

Although the numerical computation of the Poincar\'e basis has been found to be accurate (see above in this section), we also investigate its influence by replacing the exact expression of the Poincar\'e basis in the previous experiments by its numerical approximation. The results are almost the same (the difference is in the order of $10^{-12}$), showing that the whole procedure gives accurate results for the uniform distribution.\\

\subsection{Further properties of Poincar\'e quadratures} 

Empirically, we find that the nodes and weights of a $n$-point Poincar\'e quadrature have the following properties, independently of the density:
\begin{enumerate}
    \item The nodes are almost -- but in general not exactly -- equal to the zeros of $\varphi_n$. The difference is not due to numerical error. It is present even if the basis functions are known analytically, such as in the case of the truncated exponential.
    \item The nodes are almost evenly spaced, but slightly skewed towards the concentration of probability mass.
    \item The weights mimic the shape of the probability density function. 
\end{enumerate}
These three observations are illustrated in Fig.~\ref{fig:nodes_weights} for the truncated exponential distribution and for a nonparametric density.

\begin{figure}[htbp]
\centering
\subfloat[Truncated exponential ($n = 8$)]
{\includegraphics[width=.49\textwidth]{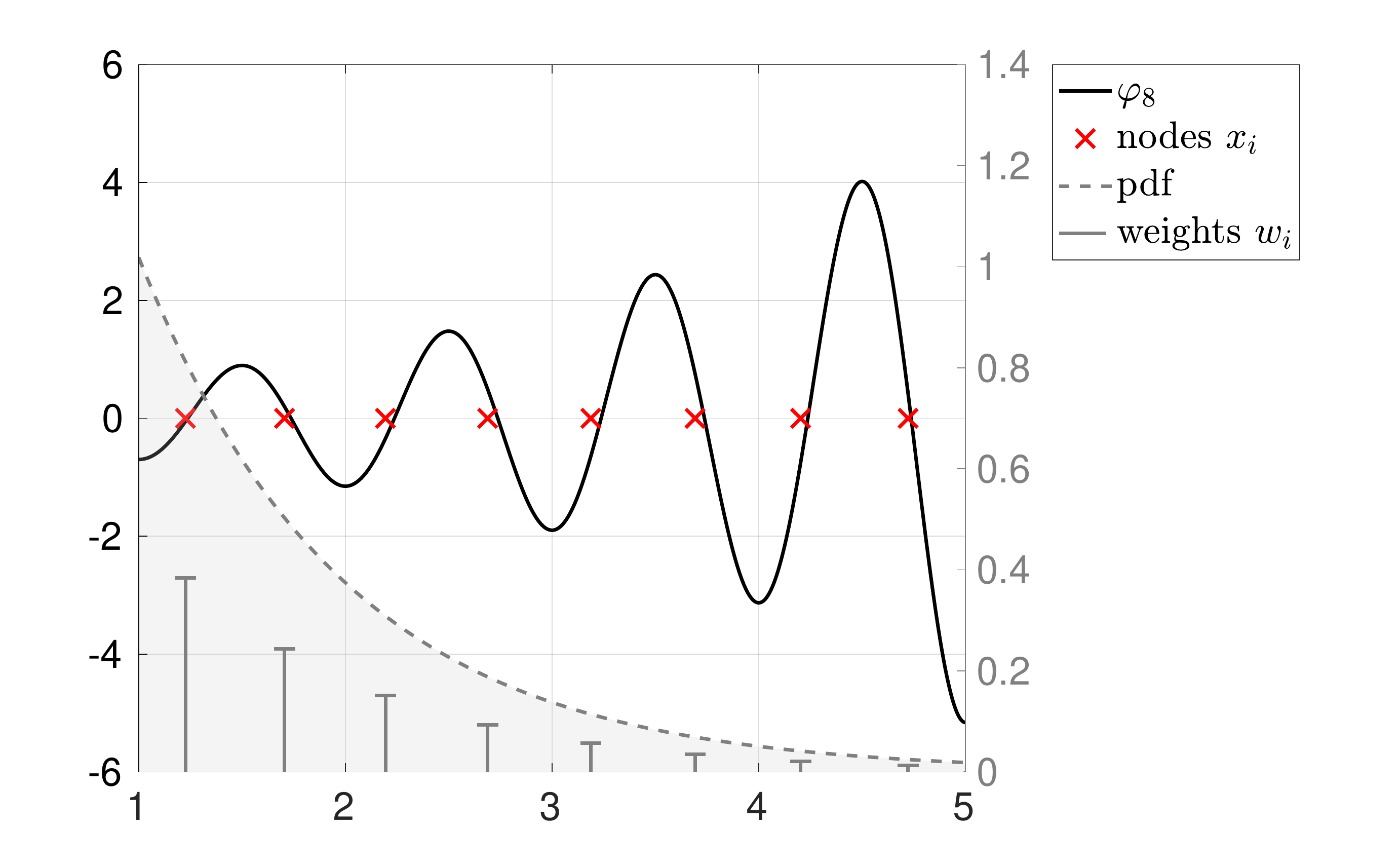}}
\hfill
\subfloat[Nonparametric density ($n = 8$)]
{\includegraphics[width=.49\textwidth]{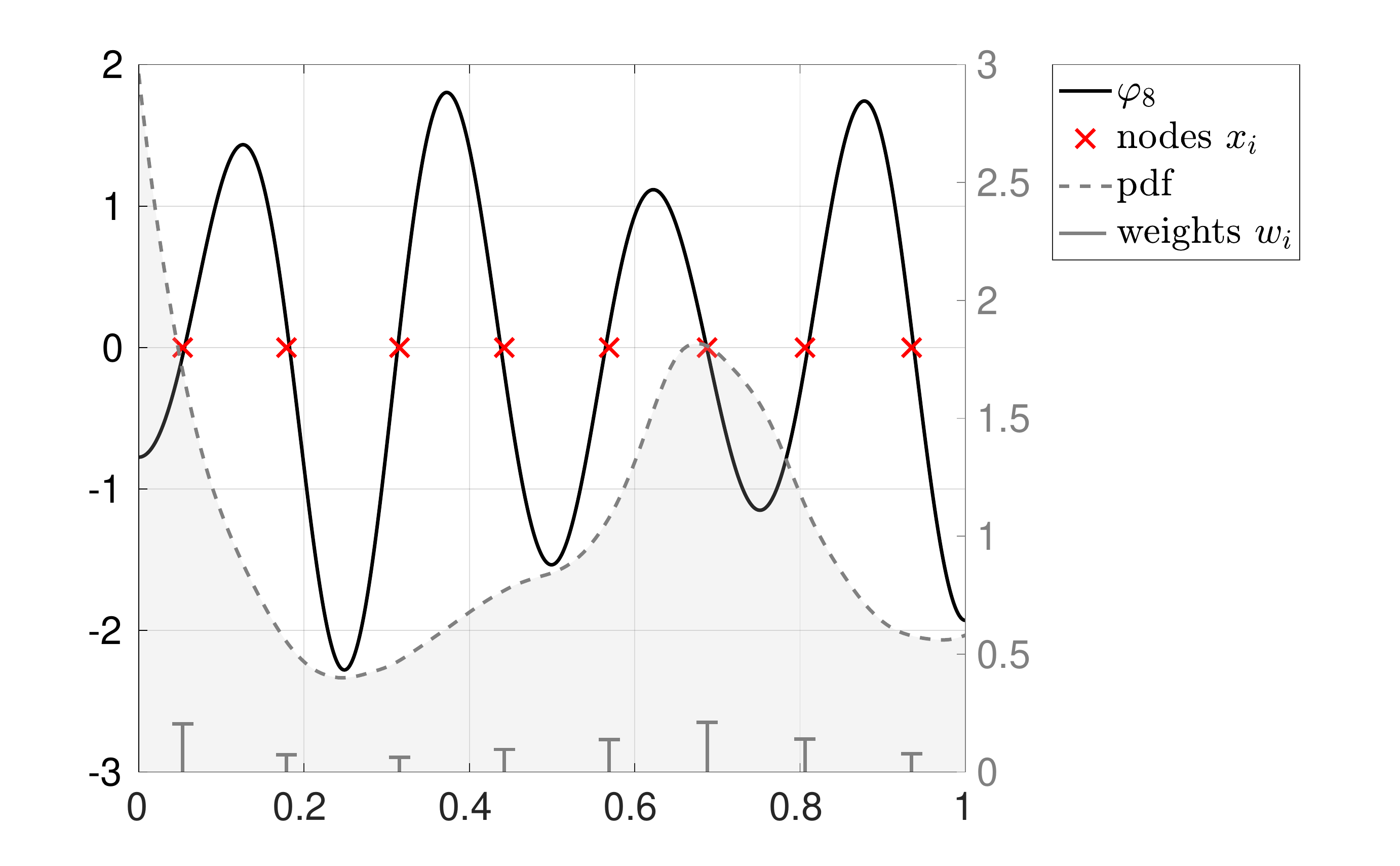}\label{fig:nodes_weights_random}}
\caption{Poincar\'e nodes and weights for the truncated exponential distribution (interval $[1,5]$, left) and a nonparametric density (right). Note that each plot has two $y$-axes with different scaling: the left one (in black) is for the basis function $\varphi_n$ and for the nodes $x_i$, while the right one (in gray) is used for the probability density and the weights $w_i$.}
\label{fig:nodes_weights}
\end{figure}

\paragraph{Experiments with random densities.}
\label{sec:quantization_experiments}
To better understand the properties of the Poincar\'e quadrature, we investigate it for a number of randomly generated continuous probability distributions. Their probability density functions (pdfs) are generated as follows.
On the interval $[0,1]$, we sample independent realizations of a Gaussian process with mean zero and Matern-$\frac{5}{2}$ covariance kernel with parameter $\theta = 0.3$. 
Denote one such realization with $g(x)$. Then we define a pdf by $\exp(g(x))$ (up to a normalization constant).
In case the minimal value on $[0,1]$ is smaller than $0.05$, we reject this pdf, in order to avoid numerical issues (recall that $\mu$ must be a bounded perturbation of the uniform distribution, and thus its pdf does not vanish on the support interval). 
One example together with its Poincar\'e quadrature for $n=8$ is shown in Fig.~\ref{fig:nodes_weights_random}.
A set of 100 such pdfs is visualized in Figure~\ref{fig:random_densities_illustration}. As expected, the configurations are various, and often provide multimodal pdfs.

\paragraph{Ratio of pdf and weights.}
To further investigate the second and third observations mentioned above, namely, that the weights of the Poincar\'e quadrature mimic the associated density, we analyse the nodes and weights associated to 100 random densities.

In Fig.~\ref{fig:random_densities_nodes}, we display boxplots of the locations of the $n=5$ nodes. We see that the nodes are nearly evenly spaced, which would correspond to the locations $(0.1, 0.3, 0.5, 0.7, 0.9)$. 

Furthermore, in Fig.~\ref{fig:random_densities_weights} we display boxplots of the ratio $\frac{n w_i}{\rho(x_i)}$ for the $n=5$ quadrature nodes, where the weights are scaled by $n$ for convenience. As already guessed from Fig.~\ref{fig:nodes_weights}, we see that this ratio is quite close to 1.
This suggests that the Poincar\'e quadrature might be a good quantization for the density $\rho$, which we investigate in the following.

\begin{figure}[htb]
\centering
\subfloat[100 randomly generated densities]
{\includegraphics[width=.44\textwidth]{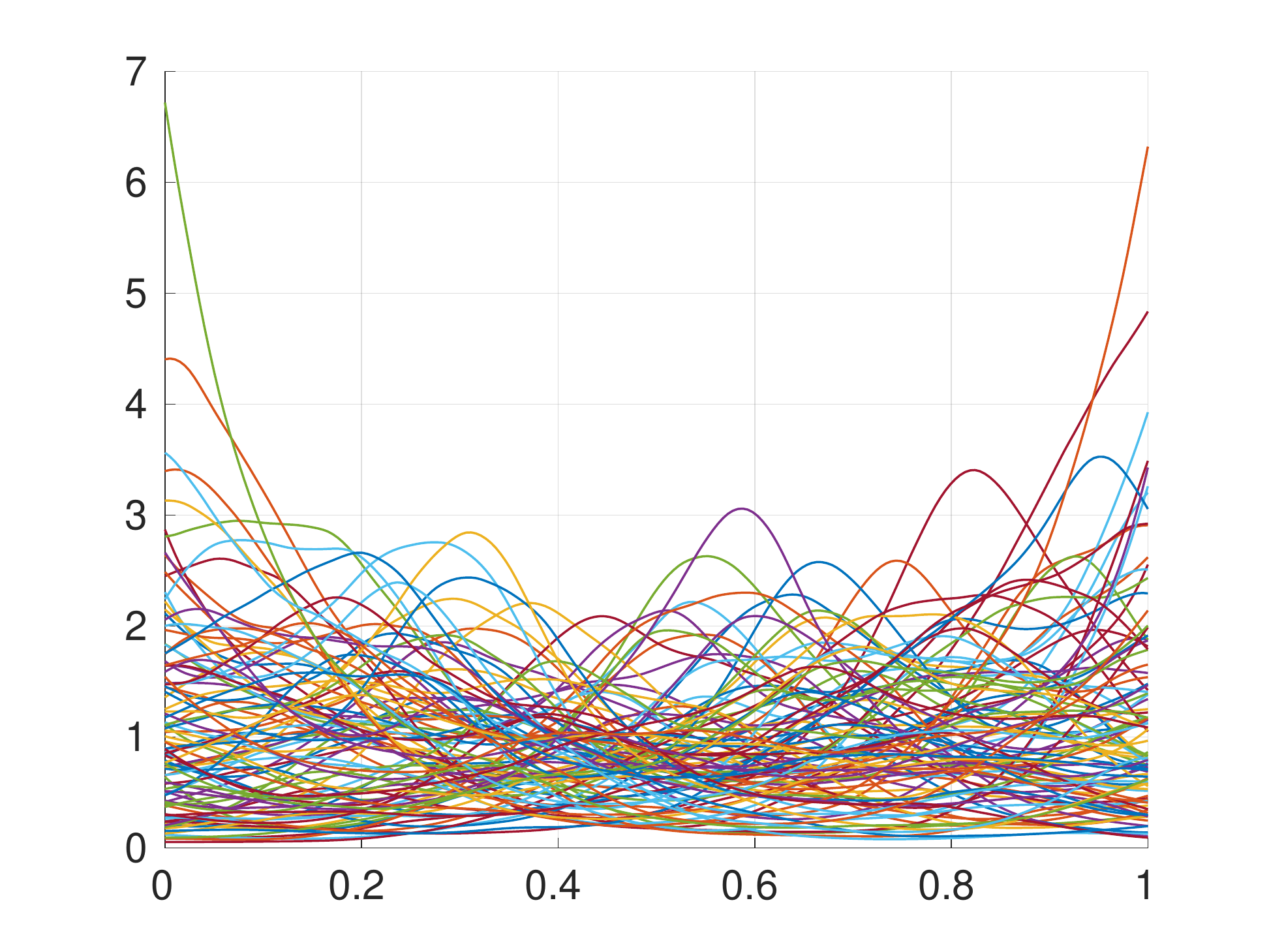}\label{fig:random_densities_illustration}}
\hfill
\subfloat[Distribution of nodes]
{\includegraphics[width=.44\textwidth]{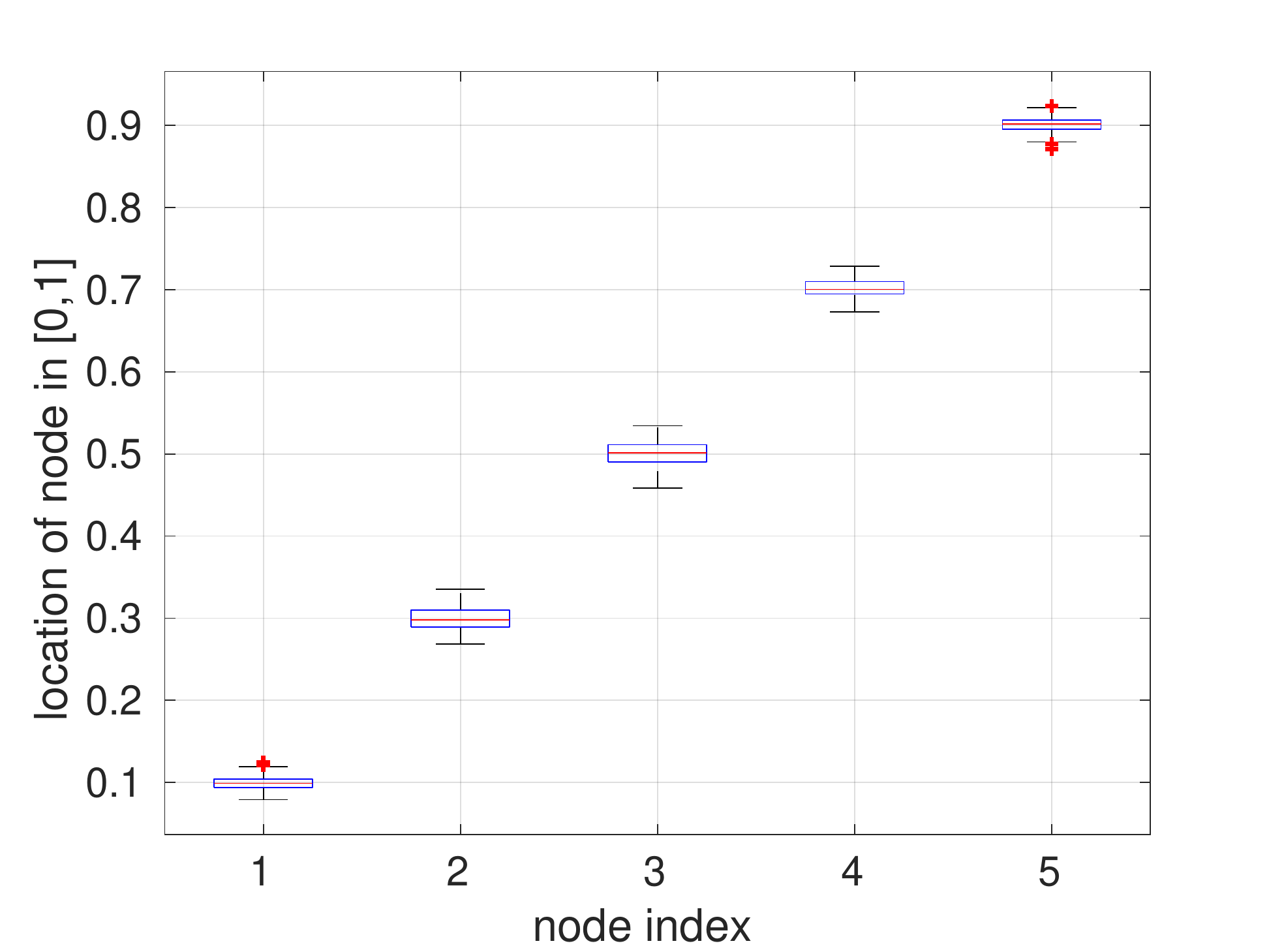}\label{fig:random_densities_nodes}}
\\
\subfloat[Ratio between weights and pdf]
{\includegraphics[width=.44\textwidth]{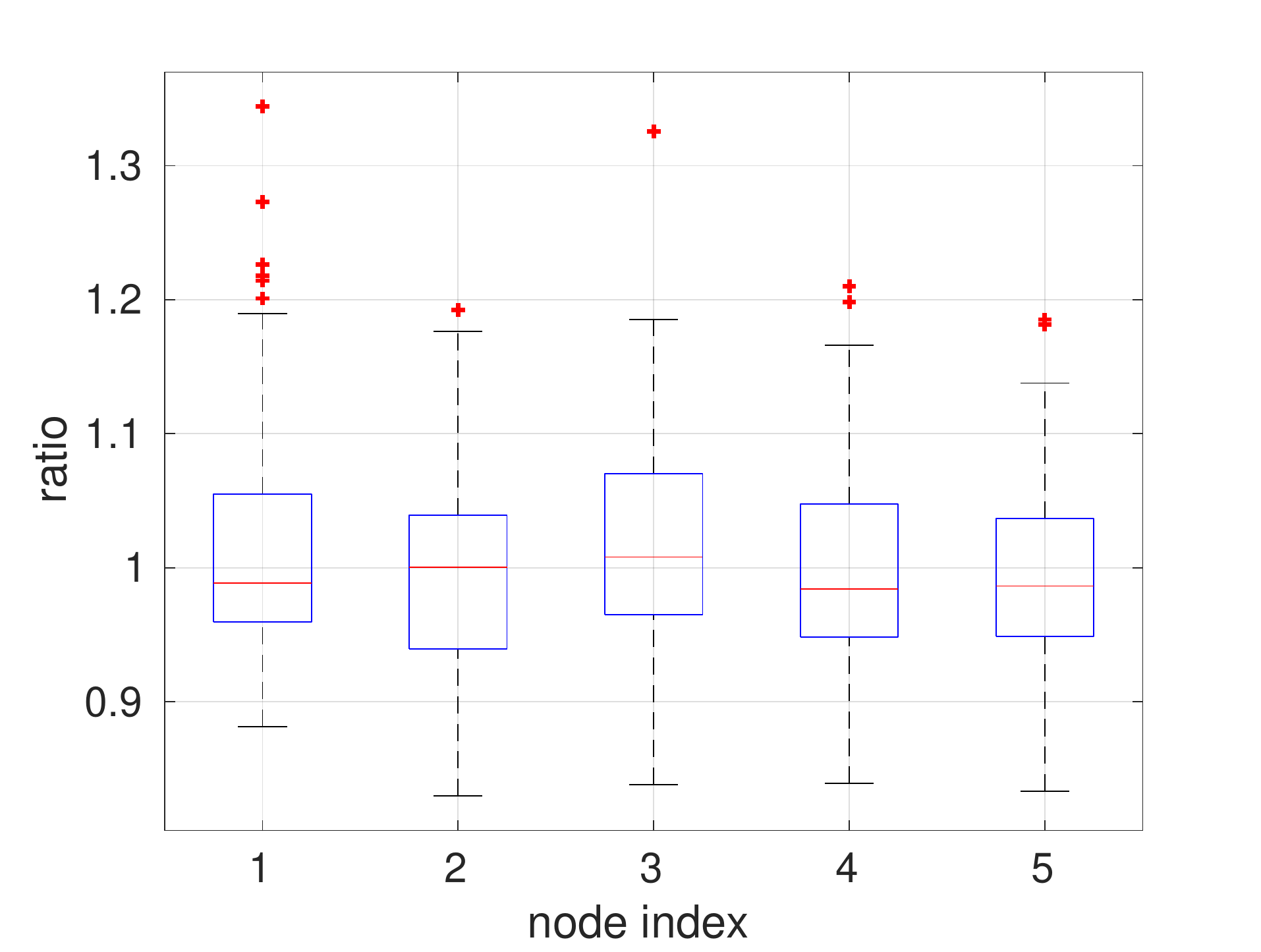}\label{fig:random_densities_weights}}
\caption{Left panel: 100 random pdfs generated by the procedure described in Section~\ref{sec:quantization_experiments}. 
Right panel: Location of nodes for the Poincar\'e quadratures associated to the same densities, with $n=4$.
Bottom panel: Ratio $\frac{w_i n}{\rho(x_i)}$, where $\rho$ is the pdf associated to $\mu$.}
\label{fig:random_densities1}
\end{figure}

\paragraph{Wasserstein-optimal quantization.}
It is interesting to compare the Poincar\'e quadrature to other standard quantization procedures, where quantization means an approximation of a continuous probability distribution by a discrete one.
Here, we will focus on the optimal quantization associated to the Wasserstein distance, called Wasserstein-optimal quantization. The Wasserstein distance between two cumulative density functions (cdf) $F, G$ is defined by
\begin{equation}
    W(F, G) = \left( \int_0^1 (F^{-1}(p) - G^{-1}(p))^2 dp \right)^{1/2}. 
\end{equation}
Then, the corresponding optimal quantization with $n$ points is the discrete probability distribution that has the smallest Wasserstein distance to the density associated to the measure $\mu$. It can be computed efficiently using Lloyd's algorithm \citep{Graf2007}.

For each random pdf, for a fixed number of nodes $n=5$, we compute the Poincar\'e quadrature, the standard Gaussian quadrature (associated to ordinary polynomials), and the Wasserstein-optimal quantization. We report the location of the nodes, as well as the zeros of $\varphi_n$, in Figure~\ref{fig:random_densities_locations}.
We observe that the Poincar\'e nodes are quite evenly spaced with small variability, and close (but not equal) to the zeros of the Poincar\'e basis function $\varphi_n$ (denoted by red crosses). 
Furthermore, the Poincar\'e nodes are more evenly spaced than the support points of the Wasserstein-optimal quantization (denoted by yellow triangles).
Finally, we observe that the Gaussian nodes are more spread out than the Poincar\'e nodes: the outermost Gaussian nodes are closer to the boundary than the outermost Poincar\'e nodes.

\begin{figure}[htb]
    \centering
     \subfloat[Locations of nodes and zeros]{\includegraphics[width=.55\textwidth, trim = 0 1.5cm 0 1.8cm, clip]{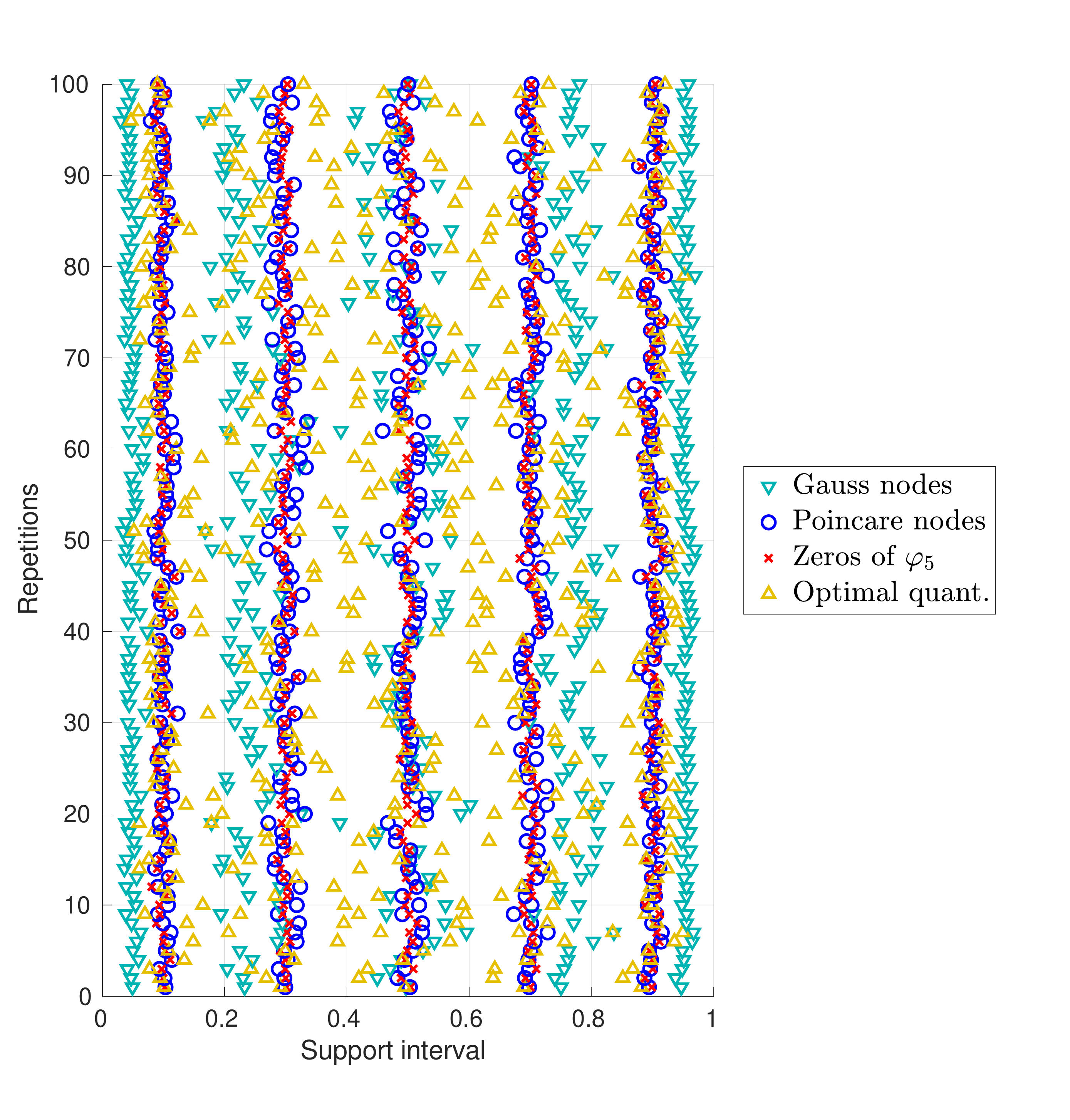} \label{fig:random_densities_locations}}
     \hfill
     \subfloat[Comparison of Wasserstein distances]{\includegraphics[width=.44\textwidth]{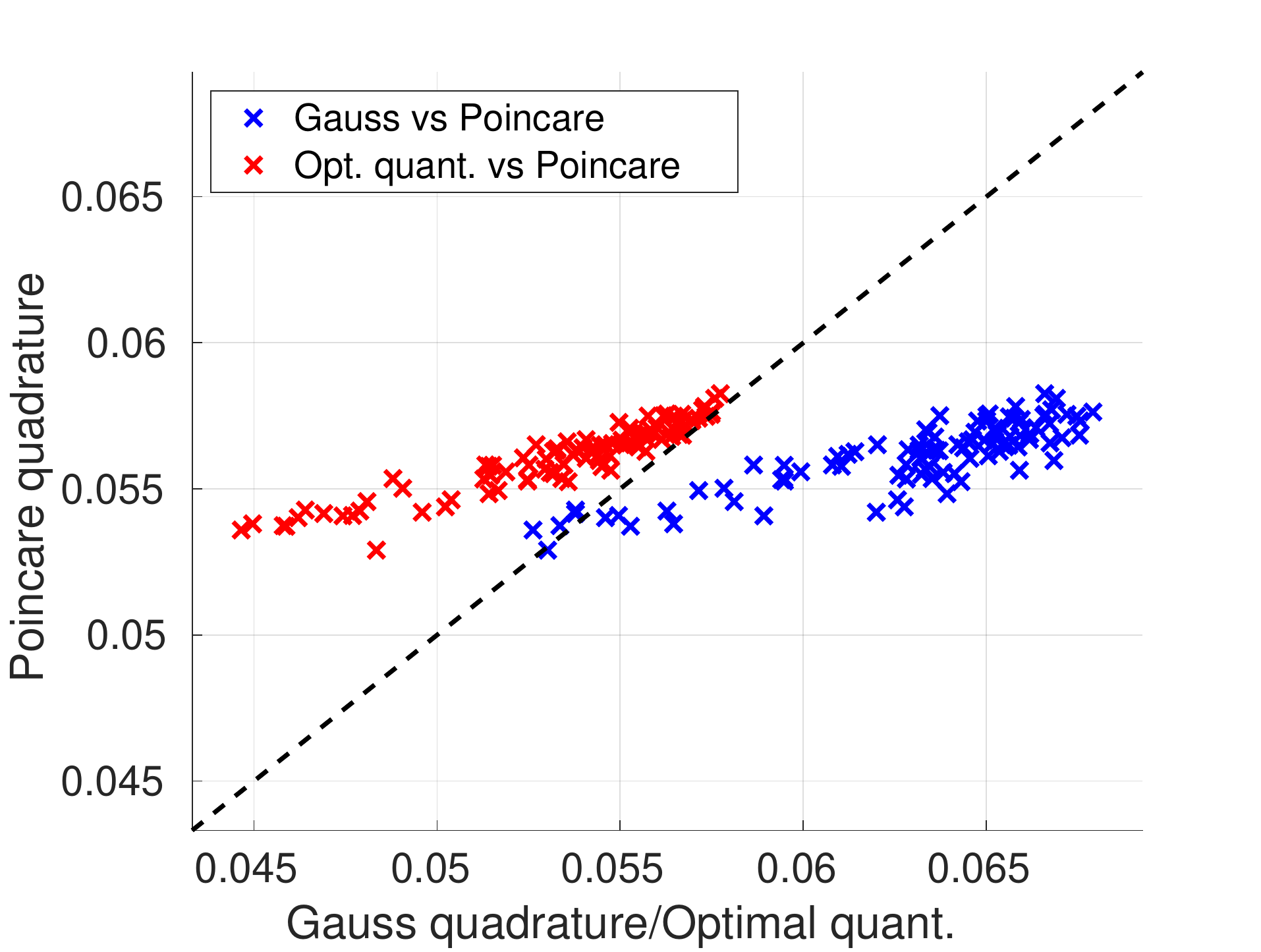}\label{fig:wasserstein_distances}}
    \caption{Analysis of the properties of Poincar\'e quadrature rules with $n=5$ nodes for the random probability distributions displayed in Fig.~\ref{fig:random_densities_illustration}.
    The left panel shows the corresponding nodes for the (usual) Gaussian quadrature and the Poincar\'e quadrature, as well as the support points for the Wasserstein-optimal quantization. The red crosses indicate the zeros of the Poincar\'e basis function $\varphi_n$.
    Right panel: Wasserstein distances between the continuous probability distribution $\mu$ and three different quadrature rules: Gaussian, Poincar\'e, and Wasserstein-optimal. Each blue point corresponds to a random density and shows the Wasserstein distances of Gaussian vs Poincar\'e quadratures. Similarly, each red point corresponds to a random density and shows the Wasserstein distances of Wasserstein-optimal vs Poincar\'e quadratures. In this way, each random density corresponds to two points in the plot. The black dashed line visualizes the identity $x = y$.
}
    \label{fig:random_densities}
\end{figure}

To further quantify the comparison, we measure the Wasserstein distance of the standard Gaussian quadrature and the Poincar\'e quadrature to $\mu$. Results with $n=5$ are presented in Figure~\ref{fig:wasserstein_distances}
We see that in most cases, the Poincar\'e quadrature has a smaller Wasserstein distance than the Gaussian one. The optimal quantization is by construction better than both, but often actually not much better than the Poincar\'e quadrature.

\subsection{Worst-case error}
We end this section by a brief analysis of the behaviour of the worst-case error $\wce(n)$ as a function of $n$. We restrict ourselves to the probability distributions for which both the kernel $K$ and the Poincar\'e basis are given explicitly, so that the only numerical error comes from the computation of the quadrature. 
We have used Eq.~\eqref{eq:wceMatricial} to compute $\wce(n)$, plugging in the Poincar\'e quadrature for the nodes $X$ and the weights $w$. Alternatively, we could have used formula \ref{eq:wce_H1_2}, which depends only on the nodes.
The curve of the worst-case error is represented in Figure~\ref{fig:wce} for the uniform distribution on $(0, 1)$ and the exponential distribution with parameter $\lambda = 1$ truncated on $(1, 5)$. 
For the uniform distribution, the result is known explicitly (Prop.~\ref{prop:quadError}), and the plot can be viewed as a validation of the numerical procedure. For the two cases, the worst-case error seems to converge at the speed of $n^{-1}$.\\

\begin{figure}[h!]
    \centering
    \includegraphics[width=.48\textwidth]{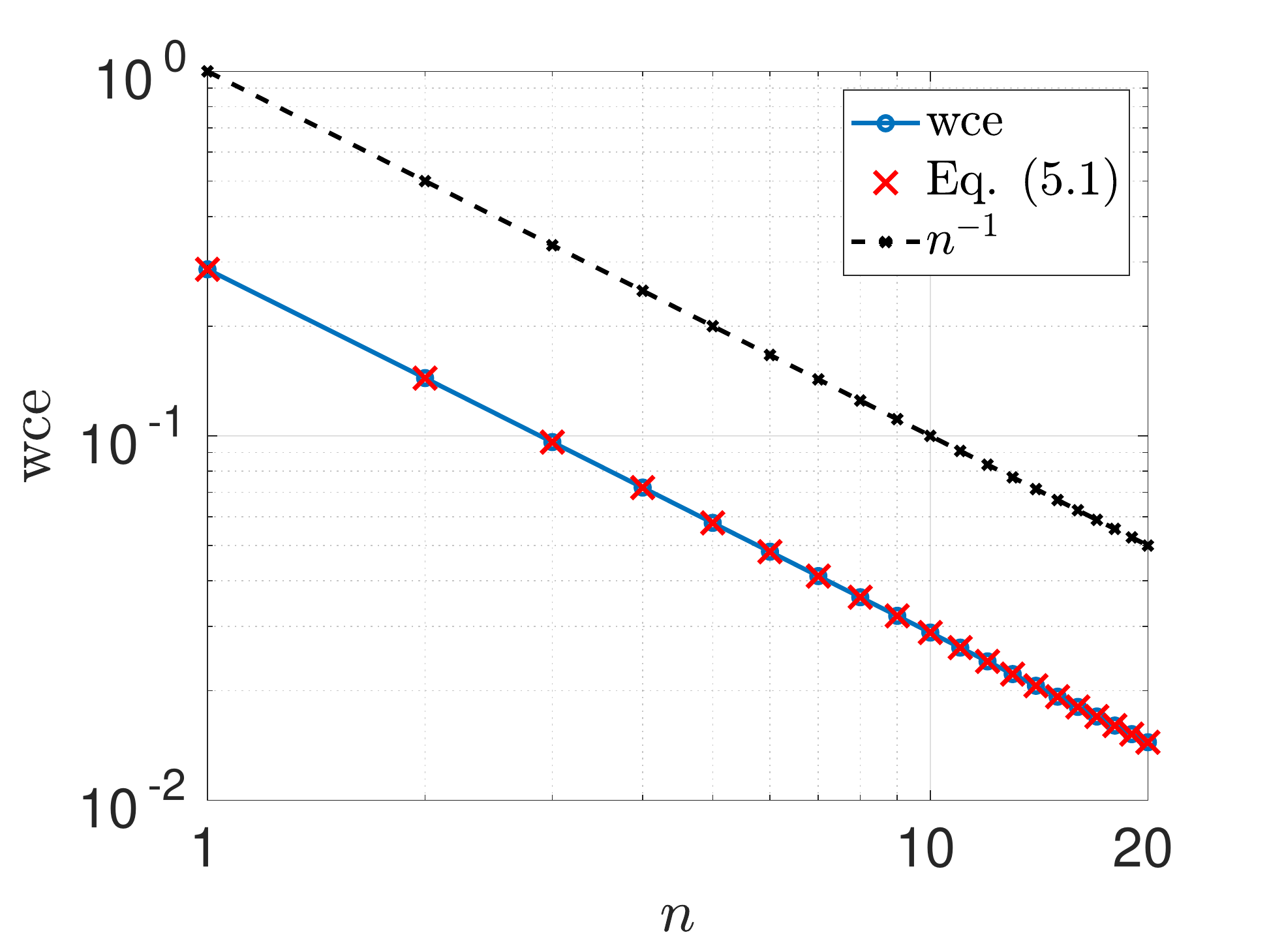} \hfill
    \includegraphics[width=.48\textwidth]{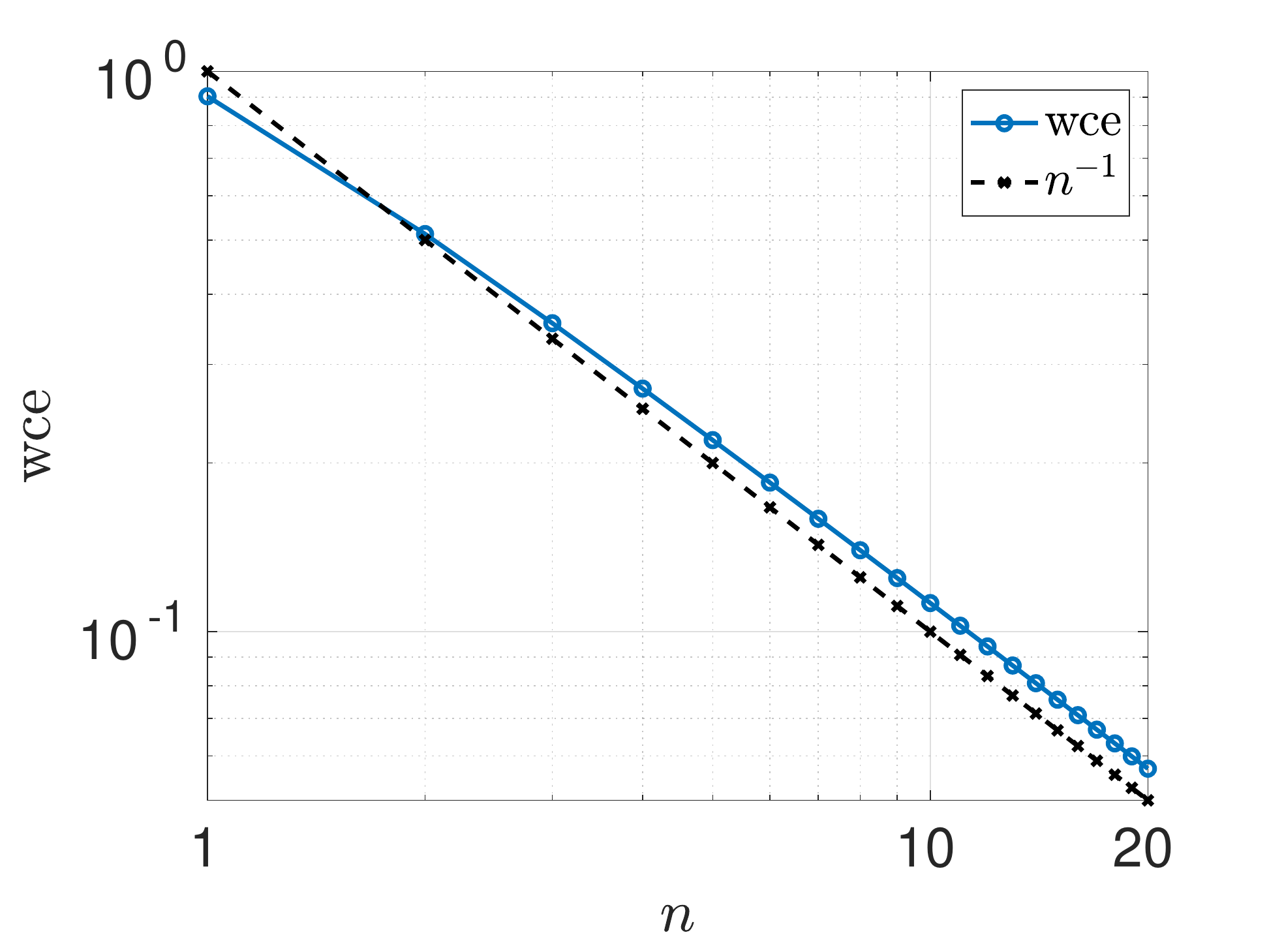} \label{fig:conv_Q_truncexp}
    \caption{Curve of the squared worst-case error $\wce(n)$, for the uniform distribution (left) and the exponential distribution with parameter $\lambda=1$, truncated on $(1,5)$ (right).}
    \label{fig:wce}
\end{figure}

\subsection*{Acknowledgement}
This research was conducted with the support of the consortium in Applied Mathematics CIROQUO, gathering partners in technological and academia in the development of advanced methods for Computer Experiments (\href{https://doi.org/10.5281/zenodo.6581217}{doi:10.5281/zenodo.6581217}) and the LabEx CIMI in the frame of the research project {\it Global sensitivity analysis and Poincar\'e inequalities}. Support from the ANR-3IA Artificial and Natural Intelligence Toulouse Institute is also gratefully
acknowledged. 
\bibliography{bibi.bib}

\begin{thebibliography}{}

\bibitem[\protect\citeauthoryear{Atteia}{Atteia}{1992}]{atteia}
Atteia, M. (1992).
\newblock {\em Hilbertian Kernels and Spline Functions}.
\newblock Studies in Computational Mathematics, 4.

\bibitem[\protect\citeauthoryear{Bakry, Gentil, and Ledoux}{Bakry
  et~al.}{2014}]{BGL_book}
Bakry, D., I.~Gentil, and M.~Ledoux (2014).
\newblock {\em Analysis and geometry of Markov diffusion operators, volume 348
  of Grundlehren der Mathematischen Wissenschaften [Fundamental Principles of
  Mathematical Sciences]}.
\newblock Springer, Cham.

\bibitem[\protect\citeauthoryear{Berlinet and Thomas-Agnan}{Berlinet and
  Thomas-Agnan}{2011}]{berlinetRKHSbook}
Berlinet, A. and C.~Thomas-Agnan (2011).
\newblock {\em Reproducing kernel {H}ilbert spaces in probability and
  statistics}.
\newblock Springer Science \& Business Media.

\bibitem[\protect\citeauthoryear{Duc-Jacquet}{Duc-Jacquet}{1973}]{these_duc_jacquet_1973}
Duc-Jacquet, M. (1973).
\newblock {\em Approximation des fonctionnelles lin{\'e}aires sur les espaces
  hilbertiens autoreproduisants}.
\newblock Ph.\ D. thesis, Universit{\'e} Joseph-Fourier-Grenoble I.

\bibitem[\protect\citeauthoryear{Gantmakher and Krejn}{Gantmakher and
  Krejn}{2002}]{oscillationBook}
Gantmakher, F.~R. and M.~G. Krejn (2002).
\newblock {\em Oscillation matrices and kernels and small vibrations of
  mechanical systems}.
\newblock Number 345. American Mathematical Soc.

\bibitem[\protect\citeauthoryear{Graf and Luschgy}{Graf and
  Luschgy}{2007}]{Graf2007}
Graf, S. and H.~Luschgy (2007).
\newblock {\em Foundations of quantization for probability distributions}.
\newblock Springer.

\bibitem[\protect\citeauthoryear{Karlin and Studden}{Karlin and
  Studden}{1966}]{karlin1966t}
Karlin, S. and W.~Studden (1966).
\newblock T-systems: with applications in analysis and statistics.
\newblock {\em Pure Appl. Math., Interscience Publishers, New York, London,
  Sidney\/}.

\bibitem[\protect\citeauthoryear{Lüthen, Roustant, Gamboa, Iooss, Marelli, and
  Sudret}{Lüthen et~al.}{2021}]{SparsePoincareChaos}
Lüthen, N., O.~Roustant, F.~Gamboa, B.~Iooss, S.~Marelli, and B.~Sudret
  (2021).
\newblock Global sensitivity analysis using derivative-based sparse {P}oincaré
  chaos expansions.

\bibitem[\protect\citeauthoryear{Oettershagen}{Oettershagen}{2017}]{oettershagenPhD2017}
Oettershagen, J. (2017).
\newblock {\em Construction of optimal cubature algorithms with applications to
  econometrics and uncertainty quantification}.
\newblock Verlag Dr. Hut.

\bibitem[\protect\citeauthoryear{Roustant, Barthe, and Iooss}{Roustant
  et~al.}{2017}]{PoincInterval}
Roustant, O., F.~Barthe, and B.~Iooss (2017).
\newblock {Poincaré inequalities on intervals – application to sensitivity
  analysis}.
\newblock {\em Electronic Journal of Statistics\/}~{\em 11\/}(2), 3081 -- 3119.

\bibitem[\protect\citeauthoryear{Ryu and Boyd}{Ryu and
  Boyd}{2015}]{Gauss_quad_LP_Ryu_Boyd_2015}
Ryu, E.~K. and S.~P. Boyd (2015).
\newblock Extensions of {G}auss quadrature via linear programming.
\newblock {\em Foundations of Computational Mathematics\/}~{\em 15\/}(4),
  953--971.

\bibitem[\protect\citeauthoryear{Szeg\"o}{Szeg\"o}{1959}]{Szego1959}
Szeg\"o, G. (1959).
\newblock Orthogonal polynomials.
\newblock In {\em Amer. Math. Soc. Colloquium, 1959}.

\bibitem[\protect\citeauthoryear{Thomas-Agnan}{Thomas-Agnan}{1996}]{thomasAgnan_H1_kernel}
Thomas-Agnan, C. (1996).
\newblock Computing a family of reproducing kernels for statistical
  applications.
\newblock {\em Numerical Algorithms\/}~{\em 13\/}(1), 21--32.

\bibitem[\protect\citeauthoryear{Zhang and Novak}{Zhang and
  Novak}{2019}]{zhang_novak_2019optimal}
Zhang, S. and E.~Novak (2019).
\newblock Optimal quadrature formulas for the {S}obolev space ${H}^1$.
\newblock {\em Journal of Scientific Computing\/}~{\em 78\/}(1), 274--289.

\end{thebibliography}

\end{document}